\documentclass[10pt]{article}
\usepackage[margin=0.8in]{geometry} 
\usepackage{amsmath,amsthm,amssymb}
\usepackage{enumitem}
\usepackage{soul}
\usepackage{mathtools}
\usepackage[utf8]{inputenc}
\usepackage{multirow}
\usepackage[colorlinks]{hyperref}
\usepackage{cleveref}
\usepackage{bbm}
\usepackage{kotex}
\usepackage{tikz-cd}
\usepackage{adjustbox}
\usepackage{algorithm, algpseudocode}
\usepackage{listings}
\usepackage{float}
\usepackage{mleftright}
\allowdisplaybreaks

\numberwithin{equation}{section}

\newcommand{\bfDelta}{\boldsymbol{\Delta}}

\newcommand{\bftheta}{\boldsymbol{\theta}}
\newcommand{\bfTheta}{\boldsymbol{\Theta}}

\newcommand{\bfomega}{\boldsymbol{\omega}}

\newcommand{\bfzero}{\mathbf{0}}

\newcommand{\bsI}{\boldsymbol{I}}

\newcommand{\bsM}{\boldsymbol{M}}

\newcommand{\bsu}{\boldsymbol{u}}

\newcommand{\bsv}{\boldsymbol{v}}

\newcommand{\bsx}{\boldsymbol{x}}
\newcommand{\bsX}{\boldsymbol{X}}

\newcommand{\bsZ}{\boldsymbol{Z}}

\newcommand{\bbB}{\mathbb{B}}

\newcommand{\bbE}{\mathbb{E}}

\newcommand{\bbN}{\mathbb{N}}

\newcommand{\bbP}{\mathbb{P}}

\newcommand{\bbR}{\mathbb{R}}

\newcommand{\bbX}{\mathbb{X}}

\newcommand{\calA}{\mathcal{A}}
\newcommand{\calB}{\mathcal{B}}
\newcommand{\calC}{\mathcal{C}}
\newcommand{\calD}{\mathcal{D}}

\newcommand{\calG}{\mathcal{G}}

\newcommand{\calI}{\mathcal{I}}
\newcommand{\calJ}{\mathcal{J}}
\newcommand{\calK}{\mathcal{K}}
\newcommand{\calL}{\mathcal{L}}

\newcommand{\calN}{\mathcal{N}}
\newcommand{\calO}{\mathcal{O}}

\newcommand{\calS}{\mathcal{S}}

\newcommand{\calX}{\mathcal{X}}

\newcommand{\sfH}{\textnormal{\textsf{H}}}

\newtheorem{thm}{Theorem}

\newtheorem{lemma}{Lemma}
\newtheorem{rmk}{Remark}

\newcommand{\RNum}[1]{\uppercase\expandafter{\romannumeral #1\relax}}

\begin{document}
 
\title{Dimension-free comparison inequalities between $f$-divergences for Gaussian location mixtures}

\author{Jeonghwan Lee\footnote{E-mail: \href{mailto:jhlee97@uchicago.edu}{\texttt{jhlee97@uchicago.edu}}.} \ and Cong Ma\footnote{E-mail: \href{mailto:congm@uchicago.edu}{\texttt{congm@uchicago.edu}}.}}

\date{
    Department of Statistics, University of Chicago \\
    \ \\
    \today
}

\maketitle

\begin{abstract}
We establish dimension-free comparison inequalities between the Kullback-Leibler (KL) and $\chi^2$-divergences and the squared Hellinger distance for Gaussian location mixtures with compactly supported mixing distributions. We show that both divergences are bounded by the squared Hellinger distance up to constants depending only on the radius of the support, independently of the ambient dimension and the number of mixture components. These results resolve an open question raised by \cite{jia2023entropic} and strengthen existing comparison inequalities. Our proof relies on a Gaussian interpolation argument that relates the derivatives of these divergences to differences between posterior means.
\end{abstract}

\section{Introduction}
\label{sec:introduction}

We establish dimension-free comparison inequalities between the Kullback-Leibler (KL) divergence, the $\chi^2$-divergence, and the squared Hellinger distance for Gaussian location mixtures. To formally state our results, let $\bbP_{\bftheta} := \calN_d \left( \bftheta, \bsI_d \right)$ denote the $d$-dimensional Gaussian distribution with mean vector $\bftheta \in \bbR^d$ and the identity covariance matrix. Given any probability measure $\pi$ supported on the Euclidean ball $\bbB_{2}^d (M) := \left\{ \bftheta \in \bbR^d : \left\| \bftheta \right\|_2 \leq M \right\}$, we define the Gaussian location mixture
\[
    \bbP_\pi(\cdot)
    := \int_{\bbB_2^d(M)}
    \bbP_{\bftheta}(\cdot)\,\pi(\mathrm{d}\bftheta).
\]
Our main result shows that the KL and $\chi^2$-divergences between two such mixtures can be upper bounded by their squared Hellinger distance, with constants depending only on $M$, independently of the ambient dimension.


\begin{thm}
\label{thm:dimension_free_comparison}
For any mixing distributions $\pi$ and $\eta$ supported on $\bbB_{2}^d (M)$,
\begin{align}
    \textnormal{\textsf{KL}} \left( \bbP_{\pi} \left\| \bbP_{\eta} \right. \right) + \textnormal{\textsf{KL}} \left( \left. \bbP_{\eta} \right\| \bbP_{\pi} \right) \leq \ & \left( 2 M^2 + 4 \right) \sfH^2 \left( \bbP_{\pi}, \bbP_{\eta} \right),
    \label{eqn:thm:dimension_free_comparison_KL_H} \\
    \max \left\{ \chi^2 \left( \bbP_{\pi} \left\| \bbP_{\eta} \right. \right), \chi^2 \left( \left. \bbP_{\eta} \right\| \bbP_{\pi} \right) \right\} \leq \ & \left\{ 18 \exp \left( 4 M^2 \right) - 2 \right\} \sfH^2 \left( \bbP_{\pi}, \bbP_{\eta} \right).
    \label{eqn:thm:dimension_free_comparison_chi_H}
\end{align}
\end{thm}

\noindent The key feature of Theorem \ref{thm:dimension_free_comparison} is that the comparison constants depend only on the radius $M$ of the mixing support, independently of the ambient
dimension $d$ or the number of mixture components. 

\begin{rmk} [Sharpness of the comparison constants]
\label{rmk:sharpness_comparison_constants} 
The constants $2 M^2 + 4$ and $18 \exp \left( 4 M^2 \right) - 2$ from the KL-$\sfH^2$ comparison inequality \eqref{eqn:thm:dimension_free_comparison_KL_H} and the $\chi^2$-$\sfH^2$ comparison inequality \eqref{eqn:thm:dimension_free_comparison_chi_H}, respectively, are \emph{sharp} in their dependence on the radius $M$ of the mixing support. Formally, we define the optimal comparison constants as
\begin{equation}
    \label{eqn:optimal_comparison_constants}
    \begin{split}
        C_{\textnormal{\textsf{KL}}} (M) := \ & \sup \left\{ \frac{\textnormal{\textsf{KL}} \left( \bbP_{\pi} \left\| \bbP_{\eta} \right. \right) + \textnormal{\textsf{KL}} \left( \left. \bbP_{\eta} \right\| \bbP_{\pi} \right)}{
        \sfH^2 \left( \bbP_{\pi}, \bbP_{\eta} \right)}: \pi \neq \eta \textnormal{ supported on } \bbB_{2}^{d} (M),\ d \in \bbN \right\}, \\
        C_{\chi^2} (M) := \ & \sup \left\{ \frac{\max \left\{ \chi^2 \left( \bbP_{\pi} \left\| \bbP_{\eta} \right. \right), \chi^2 \left( \left. \bbP_{\eta} \right\| \bbP_{\pi} \right) \right\}}{
        \sfH^2 \left( \bbP_{\pi}, \bbP_{\eta} \right)}: \pi \neq \eta \textnormal{ supported on } \bbB_{2}^{d} (M),\ d \in \bbN \right\}.
    \end{split}
\end{equation}
Then, Theorem \ref{thm:dimension_free_comparison} deduces $C_{\mathrm{KL}} (M) \leq 2 M^2 + 4$ and $C_{\chi^2} (M) \leq 18  \exp \left( 4 M^2 \right) - 2$. In order to establish lower bounds on the optimal comparison constants, take $d = 1$ and mixing distributions $\pi = \delta_{\left\{ M \right\}}$ and $\eta = \delta_{\left\{ - M \right\}}$ so that
\[
    \begin{split}
        \bbP_{\pi} = \calN \left( M, 1 \right), \qquad \bbP_{\eta} = \calN \left( -M, 1 \right).
    \end{split}
\]
Direct calculations give
\[
    \begin{split}
        \textnormal{\textsf{KL}} \left( \bbP_{\pi} \left\| \bbP_{\eta} \right. \right) + \textnormal{\textsf{KL}} \left( \left. \bbP_{\eta} \right\| \bbP_{\pi} \right) = \ & 4 M^2, \\
        \max \left\{ \chi^2 \left( \bbP_{\pi} \left\| \bbP_{\eta} \right. \right), \chi^2 \left( \left. \bbP_{\eta} \right\| \bbP_{\pi} \right) \right\} = \ & \exp \left( 4 M^2 \right) - 1, \\
        \sfH^2 \left( \bbP_{\pi}, \bbP_{\eta} \right) = \ & 2 \left\{ 1 - \exp \left( - \frac{M^2}{2} \right) \right\},
    \end{split}
\]
thereby we obtain
\begin{equation}
    \label{eqn:optimal_comparison_constants_bounds_KL}
    \begin{split}
        \frac{2 M^2}{1 - \exp \left( - \frac{M^2}{2} \right)} \leq C_{\textnormal{\textsf{KL}}} (M) \leq 2 M^2 + 4
    \end{split}
\end{equation}
and
\begin{equation}
    \label{eqn:optimal_comparison_constants_bounds_chi}
    \begin{split}
        \frac{\exp \left( 4 M^2 \right) - 1}{2 \left\{ 1 - \exp \left( - \frac{M^2}{2} \right) \right\}} \leq C_{\chi^2} (M) \leq 18 \exp \left( 4 M^2 \right) - 2.
    \end{split}
\end{equation}
With the bounds \eqref{eqn:optimal_comparison_constants_bounds_KL} and \eqref{eqn:optimal_comparison_constants_bounds_chi} on the optimal comparison constants, one can make the following observations:
\begin{itemize}
    \item The KL-$\sfH^2$ comparison inequality \eqref{eqn:thm:dimension_free_comparison_KL_H} is asymptotically sharp with the leading constant $2$: because $C_{\textnormal{\textsf{KL}}} (M) = 2 M^2 + \calO (1)$ as $M \to +\infty$, the upper bound of \eqref{eqn:optimal_comparison_constants_bounds_KL} has the optimal leading coefficient $2$. In particular, the quadratic dependence on $M$ cannot be improved.
    \item The $\chi^2$-$\sfH^2$ comparison constant from \eqref{eqn:thm:dimension_free_comparison_chi_H} is optimal up to a constant factor: since the lower bound of \eqref{eqn:optimal_comparison_constants_bounds_chi} is asymptotic to $\frac{1}{2} \exp \left( 4 M^2 \right)$ as $M \to +\infty$, it follows that $\log C_{\chi^2} (M) = 4 M^2 + \calO (1)$ as $M \to +\infty$, implying that the coefficient $4$ in the exponent is exact. Furthermore, the upper and lower bounds from \eqref{eqn:optimal_comparison_constants_bounds_chi} differ by a constant factor $36$, so the result is sharp up to a constant factor.
\end{itemize}
\end{rmk}

\indent Our results resolve an open question raised by \cite{jia2023entropic}, who established the equivalence of the KL-divergence, squared Hellinger distance, and $\chi^2$-divergence for compactly supported Gaussian location mixtures, but with the  comparison constants depending on both $M$ and $d$. They further obtained a dimension-free KL-Hellinger comparison inequality at the expense of an additional
logarithmic factor:
\[
    \textsf{KL} \left( \bbP_{\pi} \left\| \bbP_{\eta} \right. \right) \lesssim_M
    \sfH^2 \left( \bbP_{\pi}, \bbP_{\eta} \right)
    \log \frac{1}{\sfH^2 \left( \bbP_{\pi}, \bbP_{\eta} \right)}.
\]
For mixing distributions supported on at most $k$ points, \cite{doss2023optimal} established a dimension-free KL--Hellinger comparison without the logarithmic factor, but with a constant growing exponentially in $k^2$. Theorem \ref{thm:dimension_free_comparison} removes
the logarithmic factor and the dependence on the number of mixture components, establishing dimension-free comparisons for both the KL and $\chi^2$ divergences.

\indent These comparison inequalities are closely related to the literature of density
estimation, where the KL-divergence, squared Hellinger distance, and $\chi^2$-divergence are widely used as performance metrics of estimators \cite{zhang2009generalized, jiang2009general, kim2014minimax, saha2020nonparametric, polyanskiy2021sharp, kim2022minimax}. In particular, the squared Hellinger distance plays an important role in Gaussian mixture density estimation, with applications to regret bounds for empirical Bayes estimation under the $g$-modeling approach \cite{jiang2009general, chen2026sharp,efron2014two, efron2024empirical}.

\section{Proof of Theorem \ref{thm:dimension_free_comparison}}
\label{sec:proofs_dimension_free_comparison_inequalities}

Our strategy for the proof of Theorem \ref{thm:dimension_free_comparison} heavily relies on an interpolation from the standard Gaussian distribution to each Gaussian location mixture. The core observation is that the derivatives of the interpolation curves induced by the KL-divergence, squared Hellinger distance, and $\chi^2$-divergence are all weighted integrals of the same squared Euclidean norm of the difference between posterior means. Comparison of these derivatives yields both comparison inequalities.

\subsection{Interpolation between the standard Gaussian and Gaussian location mixtures}
\label{subsec:interpolation_standard_gaussian_gaussian_location_mixtures}

\indent Let $\phi_{d} (\bsx) := \left( 2 \pi \right)^{- \frac{d}{2}} \exp \left( - \frac{\left\| \bsx \right\|_{2}^{2}}{2} \right)$ denote the density function of the $d$-dimensional standard Gaussian distribution. We interpolate between $\phi_d$ and the density functions of the two Gaussian location mixtures via scaling the mixing vectors: for $t \in \bbR$, we define
\begin{equation}
    \label{eqn:interpolation_standard_gaussian_gaussian_location_mixtures_v1}
    \begin{split}
        p_{t} (\bsx) := \ & \int_{\bbB_{2}^d (M)} \phi_{d} \left( \bsx - t \bftheta \right) \pi \left( \mathrm{d} \bftheta \right), \\
        q_{t} (\bsx) := \ & \int_{\bbB_{2}^d (M)} \phi_{d} \left( \bsx - t \bftheta \right) \eta \left( \mathrm{d} \bftheta \right).
    \end{split}
\end{equation}
Then, we obtain $p_0 = q_0 = \phi_d$, while $p_1$ and $q_1$ are the density functions of $\bbP_{\pi}$ and $\bbP_{\eta}$, respectively. Equivalently, $p_t$ and $q_t$ are the density functions of $\bsX (t) = t \bfTheta + \bsZ$, when $\bfTheta \sim \pi$ and $\bfTheta \sim \eta$, respectively, along with $\bsZ \sim \calN_{d} \left( \bfzero_d, \bsI_d \right)$ independent of $\bfTheta$. Let $\bbE_{\pi}$ and $\bbE_{\eta}$ denote the expectations under these two mixing distributions.

\indent Hereafter, let us adopt the same divergence notations for density functions and their corresponding probability measures. Along the interpolation, we define for $t \in \bbR$:
\begin{equation}
    \label{eqn:interpolation_standard_gaussian_gaussian_location_mixtures_v2}
    \begin{split}
        J (t) := \ & \textsf{KL} \left( p_t \left\| q_t \right. \right) + \textsf{KL} \left( q_t \left\| p_t \right. \right), \\
        C(t):= \ & \chi^2 \left( p_t \left\| q_t \right. \right), \\
        D (t) := \ & \sfH^2 \left( p_t, q_t \right).
    \end{split}
\end{equation}
Then, our goal is to upper bound $J (1)$ and $C (1)$ in terms of $D (1)$. Since $p_0 = q_0$, we have $J (0) = D (0) = C (0) = 0$, so one can obtain the desired bounds by comparing the derivatives along the interpolation and integrating from $0$ to $1$. The derivative formulas involve the posterior means of the original mixing vector, which we define by
\begin{equation}
    \label{eqn:interpolation_standard_gaussian_gaussian_location_mixtures_v3}
    \begin{split}
        \bsu_t (\bsx) := \bbE_{\pi} \left[ \bfTheta \mid \bsX (t) = \bsx \right], \qquad \bsv_t (\bsx) := \bbE_{\eta} \left[ \bfTheta \mid \bsX (t) = \bsx \right].
    \end{split}
\end{equation}
Since both the posterior distributions are supported on $\bbB_{2}^d (M)$, their means should satisfy
\begin{equation}
    \label{eqn:interpolation_standard_gaussian_gaussian_location_mixtures_v4}
    \left\| \bsu_t (\bsx) \right\|_{2} \leq  M, \qquad \left\| \bsv_t (\bsx) \right\|_{2} \leq  M.
\end{equation}
For each $t \in \bbR$, set $\bfDelta_t := \bsu_t - \bsv_t : \bbR^d \to \bbR^d$. In particular, $\left\| \bfDelta_t (\bsx) \right\|_{2} \leq 2M$ for every $\left( t, \bsx \right) \in \bbR \times \bbR^d$.

\paragraph{Divergence derivatives.}
The following lemma provides expressions of all three derivatives as weighted integrals of $\left\| \bfDelta_t (\bsx) \right\|_{2}^2$, whose detailed proof is deferred to Appendix \ref{sec:proof_lemma:proof_thm:dimension_free_comparison_divergence_derivatives}.

\begin{lemma} [Divergence derivatives]
\label{lemma:proof_thm:dimension_free_comparison_divergence_derivatives}
The functions $J$, $C$, and $D$ defined in \eqref{eqn:interpolation_standard_gaussian_gaussian_location_mixtures_v2} are all finite and belong to $\calC^1 (\bbR)$. Their derivatives are
\begin{equation}
    \label{eqn:lemma:proof_thm:dimension_free_comparison_divergence_derivatives_v1}
    \begin{split}
        J' (t) = \ & t \int_{\bbR^d} \left\{ p_t (\bsx) + q_t (\bsx) \right\} \left\| \bfDelta_t (\bsx) \right\|_{2}^2 \mathrm{d} \bsx, \\
        C' (t) = \ & 2t \int_{\bbR^d} \frac{p_{t}^2 (\bsx)}{q_{t} (\bsx)} \left\| \bfDelta_t (\bsx) \right\|_{2}^2 \mathrm{d} \bsx, \\
        D' (t) = \ & \frac{t}{2} \int_{\bbR^d} \sqrt{p_t (\bsx) q_t (\bsx)} \left\| \bfDelta_t (\bsx) \right\|_{2}^2 \mathrm{d} \bsx.
    \end{split}
\end{equation}
\end{lemma}

\noindent The derivative formula for $D$ from \eqref{eqn:lemma:proof_thm:dimension_free_comparison_divergence_derivatives_v1} implies $D' (t) \geq 0$ for $t \geq 0$; in particular, the function $D$ is non-decreasing on $[0, +\infty)$ and $D (t) \leq D (1)$ for $t \in [0, 1]$. 

\indent At a formal level, the derivative identities follow from a short
integration-by-parts calculation using the posterior mean representations of the score functions. We present this calculation in Appendix \ref{subsec:proof_divergence_derivative_formulas}, and provide the analytic justifications, including differentiation under the integral sign and control of boundary terms, in Appendix \ref{subsec:analytic_justifications}.

\begin{rmk}
\label{rmk:lemma:proof_thm:dimension_free_comparison_divergence_derivatives}
The derivative identities from Lemma \ref{lemma:proof_thm:dimension_free_comparison_divergence_derivatives} are particular cases of the generalized de Bruijn's identity for $f$-divergences; see \emph{Proposition 6} of \cite{valero2017generalization}. Applying this result requires identifying the PDE exhibited by the interpolation densities and verifying some local integrability and boundary-decay conditions. We discuss this connection in Section~\ref{sec:connection_de_bruijn_identity}.
\end{rmk}

\subsection{Proof of the KL-divergence comparison inequality}
\label{subsec:proof_eqn:thm:dimension_free_comparison_KL_H}

\indent To establish the dimension-free $\textsf{KL}$-$\sfH^2$ comparison inequality \eqref{eqn:thm:dimension_free_comparison_KL_H}, it is sufficient to prove $J (1) \leq \left( 2 M^2 + 4 \right) D (1)$. Towards this end, we compare $J' (t)$ with $D' (t)$ by leveraging the posterior mean bound \eqref{eqn:interpolation_standard_gaussian_gaussian_location_mixtures_v4} from Section \ref{subsec:interpolation_standard_gaussian_gaussian_location_mixtures}, and then use the monotonicity of $D$. The pointwise identity
\[
    \begin{split}
        p_t (\bsx) + q_t (\bsx) = \left\{ \sqrt{p_t (\bsx)} - \sqrt{q_t (\bsx)} \right\}^2 + 2 \sqrt{p_t (\bsx) q_t (\bsx)}, \qquad \left( t, \bsx \right) \in \bbR \times \bbR^d,
    \end{split}
\]
and the derivative identities from \eqref{eqn:lemma:proof_thm:dimension_free_comparison_divergence_derivatives_v1} give
\begin{equation}
    \label{eqn:proof_eqn:thm:dimension_free_comparison_KL_H_v1}
    \begin{split}
        J' (t) = \ & t \int_{\bbR^d} \left\{ \sqrt{p_t (\bsx)} - \sqrt{q_t (\bsx)} \right\}^2 \underbrace{\left\| \bfDelta_t (\bsx) \right\|_{2}^2}_{\leq \ 4 M^2} \mathrm{d} \bsx + 2t \int_{\bbR^d} \sqrt{p_t (\bsx) q_t (\bsx)} \left\| \bfDelta_t (\bsx) \right\|_{2}^2 \mathrm{d} \bsx \\
        \stackrel{\textnormal{(a)}}{\leq} \ & 4 M^2 t D (t) + 4 D' (t)
    \end{split}
\end{equation}
for every $t \geq 0$, where the step (a) invokes the fact $\left\| \bfDelta_t (\bsx) \right\|_{2} \leq 2M$. Define
\[
    \begin{split}
        F_{\textsf{KL}} (t) := J (t) - 4 D (t) - 2 M^2 t^2 D (1), \qquad t \in \bbR.
    \end{split}
\]
Then, it holds for every $t \in [0, 1]$ that
\begin{equation}
    \label{eqn:proof_eqn:thm:dimension_free_comparison_KL_H_v2}
    \begin{split}
        F_{\textsf{KL}}' (t) = J' (t) - 4 D' (t) - 4 M^2 t D (1) \stackrel{\textnormal{(b)}}{\leq} 4 M^2 t \left\{ D (t) - D (1) \right\} \stackrel{\textnormal{(c)}}{\leq} 0,
    \end{split}
\end{equation}
where the step (b) holds due to the inequality \eqref{eqn:proof_eqn:thm:dimension_free_comparison_KL_H_v1}, and the step (c) follows since $D (t) \leq D (1)$ for $t \in [0, 1]$, which comes from the monotonicity of $D$ on $[0, +\infty)$. Therefore, $F_{\textsf{KL}}$ is non-increasing on $[0, 1]$ and
\begin{equation}
    \label{eqn:proof_eqn:thm:dimension_free_comparison_KL_H_v3}
    \begin{split}
        F_{\textsf{KL}} (1) = J (1) - \left( 2 M^2 + 4 \right) D (1) \leq F_{\textsf{KL}} (0) = 0,
    \end{split}
\end{equation}
concluding the proof of the $\textsf{KL}$-$\sfH^2$ comparison inequality \eqref{eqn:thm:dimension_free_comparison_KL_H}.

\subsection{Proof of the chi-squared divergence comparison inequality}
\label{subsec:proof_eqn:thm:dimension_free_comparison_chi_H}

\indent In order to upper bound $C (1) = \chi^2 \left( \bbP_{\pi} \left\| \bbP_{\eta} \right. \right)$ in terms of $D (1) = \sfH^2 \left( \bbP_{\pi}, \bbP_{\eta} \right)$, we mainly work with the interpolation curve $G := C + 2D$ in lieu of directly comparing $C'$ with $D'$. This combination allows us to control $G' (t)$ in terms of $G (t)$ and $D' (t)$. We claim that
\begin{equation}
\label{eqn:proof_proof_eqn:thm:dimension_free_comparison_chi_H_v1}
    \begin{split}
        G' (t) \leq 8 M^2 t G (t) + 18 D' (t), \qquad t \in [0, +\infty)
    \end{split}
\end{equation}
We first show how this differential inequality completes the proof of the $\chi^2$-$\sfH^2$ comparison \eqref{eqn:thm:dimension_free_comparison_chi_H}, and then verify it. As in the proof of the \textsf{KL}-$\sfH^2$ comparison inequality \eqref{eqn:thm:dimension_free_comparison_KL_H}, we construct a non-increasing function. The multiplicative factor $\exp \left( - 4 M^2 t^2 \right)$ cancels the term $8 M^2 t G(t)$, so we define
\[
    \begin{split}
        F_{\chi^2} (t) := \exp \left( - 4 M^2 t^2 \right) G (t) - 18 D (t), \qquad t \in \bbR.
    \end{split}
\]
Then, it follows by the differential inequality \eqref{eqn:proof_proof_eqn:thm:dimension_free_comparison_chi_H_v1} alongside with $D' (t) \geq 0$ for every $t \geq 0$ that
\[
    \begin{split}
        F_{\chi^2}' (t) = \exp \left( - 4 M^2 t^2 \right) \left\{ G' (t) - 8 M^2 t G (t) \right\} - D' (t) \leq - 18 \left\{ 1 - \exp \left( - 4 M^2 t^2 \right) \right\} D' (t) \leq 0
    \end{split}
\]
for every $t \geq 0$. Therefore, the function $F_{\chi^2}$ is non-increasing on $[0, +\infty)$. Since $G (1) = C (1) + 2 D (1)$ and $G (0) = D (0) = 0$, we obtain
\begin{equation}
    \label{eqn:proof_proof_eqn:thm:dimension_free_comparison_chi_H_v2}
    \begin{split}
        F_{\chi^2} (1) = \exp \left( - 4 M^2 \right) \left\{ C (1) + 2 D (1) \right\} - 18 D (1) \leq F_{\chi^2} (0) = 0.
    \end{split}
\end{equation}
Rearranging the bound \eqref{eqn:proof_proof_eqn:thm:dimension_free_comparison_chi_H_v2} yields the desired comparison inequality:
\begin{equation}
    \label{eqn:proof_proof_eqn:thm:dimension_free_comparison_chi_H_v3}
    \begin{split}
        \chi^2 \left( \bbP_{\pi} \left\| \bbP_{\eta} \right. \right) = C (1) \leq \left\{ 18 \exp \left( 4 M^2 \right) - 2 \right\} D (1) = \left\{ 18 \exp \left( 4 M^2 \right) - 2 \right\} \sfH^2 \left( \bbP_{\pi}, \bbP_{\eta} \right).
    \end{split}
\end{equation}
Interchanging $\pi$ and $\eta$ produces the reverse $\chi^2$-$\sfH^2$ comparison bound $\chi^2 \left( \bbP_{\eta} \left\| \bbP_{\pi} \right. \right) \leq  \left\{ 18 \exp \left( 4 M^2 \right) - 2 \right\} \sfH^2 \left( \bbP_{\pi}, \bbP_{\eta} \right)$
due to the symmetry of the squared Hellinger distance. Hence, the differential inequality \eqref{eqn:proof_proof_eqn:thm:dimension_free_comparison_chi_H_v1} implies the desired $\chi^2$-$\sfH^2$ comparison bound \eqref{eqn:thm:dimension_free_comparison_chi_H}.

\indent Thus, it remains to prove the differential inequality \eqref{eqn:proof_proof_eqn:thm:dimension_free_comparison_chi_H_v1}. Write $\rho_{t} := \frac{p_{t}}{q_{t}}$ for $t \in \bbR$ and consider the non-negative integrand $\calG : \bbR \times \bbR^d \to [0, +\infty)$, where
\[
    \begin{split}
        \calG \left( t, \bsx \right) := \ & \frac{p_{t}^2 (\bsx)}{q_{t} (\bsx)} - 4 \sqrt{p_{t} (\bsx) q_{t} (\bsx)} + 3 \\
        = \ & q_{t} (\bsx) \left\{ \sqrt{\rho_{t} (\bsx)} - 1 \right\}^2 \left\{ \rho_{t} (\bsx) + 2 \sqrt{\rho_{t} (\bsx)} + 3 \right\} \geq 0
    \end{split}
\]
Then, the definitions of $C$ and $D$ together with the fact $\int_{\bbR^d} q_{t} (\bsx) \mathrm{d} \bsx = 1$ gives
\begin{equation}
    \label{eqn:proof_proof_eqn:thm:dimension_free_comparison_chi_H_v4}
    \begin{split}
        \int_{\bbR^d} \calG \left( t, \bsx \right) \mathrm{d} \bsx = C (t) + 2 D (t) = G (t), \qquad t \in \bbR.
    \end{split}
\end{equation}
By virtue of the derivative formulas for $C$ and $D$ from \eqref{eqn:lemma:proof_thm:dimension_free_comparison_divergence_derivatives_v1}, one can reveal for every $t \geq 0$ that
\begin{equation}
    \label{eqn:proof_proof_eqn:thm:dimension_free_comparison_chi_H_v5}
    \begin{split}
        G' (t) = \ & t \int_{\bbR^d} \left\{ 2 \calG \left( t, \bsx \right) + 9 \sqrt{p_{t} (\bsx) q_{t} (\bsx)} - 6 q_{t} (\bsx) \right\} \left\| \bfDelta_{t} (\bsx) \right\|_{2}^2 \mathrm{d} \bsx \\
        \stackrel{\textnormal{(a)}}{\leq} \ & 2t \int_{\bbR^d} \calG \left( t, \bsx \right) \underbrace{\left\| \bfDelta_{t} (\bsx) \right\|_{2}^2}_{\leq \ 4 M^2} \mathrm{d} \bsx + 18 D' (t) \\
        \stackrel{\textnormal{(b)}}{\leq} \ & 8 M^2 t G (t) + 18 D' (t),
    \end{split}
\end{equation}
where the step (a) holds by dropping the non-positive integral term involving $- 6 q_{t} (\bsx)$ and leveraging the derivative formula for $D$ from \eqref{eqn:lemma:proof_thm:dimension_free_comparison_divergence_derivatives_v1}, and the step (b) is a consequence of the facts $\calG \geq 0$ together with $\left\| \bfDelta_{t} (\bsx) \right\|_{2} \leq 2M$. This establishes the desired differential inequality \eqref{eqn:proof_proof_eqn:thm:dimension_free_comparison_chi_H_v1}, and therefore completes the proof of Theorem \ref{thm:dimension_free_comparison}.

\section{Connections to the generalized de Bruijn's identity}
\label{sec:connection_de_bruijn_identity}

Let $p : \bbR^d \to [0, +\infty)$ and $q : \bbR^d \to [0, +\infty)$ be two density functions. For any given twice continuously differentiable convex function $f : [0, +\infty) \to \bbR$, the \emph{$f$-divergence} between $p$ and $q$ is defined as
\[
    \begin{split}
        \calD_{f} \left( p \| q \right) := \int_{\bbR^d} q (\bsx) f \left\{ \frac{p (\bsx)}{q (\bsx)} \right\} \mathrm{d} \bsx,
    \end{split}
\]
and the associated spatial Fisher divergence matrix is given by
\begin{equation}
    \label{eqn:spatial_fisher_divergence_matrix}
    \begin{split}
        \calI_{f} \left( p \| q \right) := \int_{\bbR^d} q (\bsx) \rho^2 (\bsx) f'' \left\{ \rho (\bsx) \right\} \cdot \nabla_{\bsx} \log \rho (\bsx) \nabla_{\bsx}^{\top} \log \rho (\bsx) \mathrm{d} \bsx, \qquad \rho (\bsx) := \frac{p (\bsx)}{q (\bsx)}.
    \end{split}
\end{equation}
Taking the trace to the equation \eqref{eqn:spatial_fisher_divergence_matrix} gives
\[
    \begin{split}
        \mathrm{Trace} \left\{ \calI_{f} \left( p \| q \right) \right\} = \int_{\bbR^d} q (\bsx) \rho^2 (\bsx) f'' \left\{ \rho (\bsx) \right\} \left\| \nabla_{\bsx} \log \rho (\bsx) \right\|_{2}^2 \mathrm{d} \bsx.
    \end{split}
\]
A straightforward calculation reveals that the interpolation families $\left\{ p_{t} : t \in \bbR \right\}$ and $\left\{ q_{t} : t \in \bbR \right\}$ from \eqref{eqn:interpolation_standard_gaussian_gaussian_location_mixtures_v1} satisfy the following PDEs:
\begin{equation}
    \label{eqn:de_bruijn_pde}
    \begin{split}
        t \cdot \partial_{t} p_{t} = - \mathrm{div}_{\bsx} \left( \bsx p_{t} \right) - \Delta_{\bsx} p_{t}, \qquad t \cdot \partial_{t} q_{t} = - \mathrm{div}_{\bsx} \left( \bsx q_{t} \right) - \Delta_{\bsx} q_{t},
    \end{split}
\end{equation}
where the symbols here mean
\[
    \begin{split}
        \mathrm{div}_{\bsx} \bfomega := \sum_{i=1}^{d} \partial_{x_i} \omega_{i}, \qquad \Delta_{\bsx} g := \mathrm{div}_{\bsx} \left( \nabla_{\bsx} g \right) = \sum_{i=1}^{d} \partial_{x_{i}}^2 g.
    \end{split}
\]
Following the notation of the equation (23) in \cite{valero2017generalization}, the corresponding linear operators are given by
\[
    \begin{split}
        \calL_{1}: \bbR \to \bbR,& \qquad \calL _{1} (a) = ta, \\
        \calL_{2} : \bbR \to \bbR,& \qquad \calL_{2} (a) = 0, \\
        \beta_{1}: \bbR^d \times \bbR \to \bbR^d,& \qquad \beta_{1} \left( \bsx, t \right) = \bsx, \\ 
        \calK_{1}: \bbR^{d \times d} \to \bbR,& \qquad \calK_{1} (\bsM) = - \mathrm{Trace} (\bsM), \\
        \calK_{2}: \bbR^{d \times d} \to \bbR,& \qquad \calK_{2} (\bsM) = - \mathrm{Trace} (\bsM).
    \end{split}
\]
Applying the generalized multivariate de Bruijn's identity for $f$-divergences (\emph{Proposition 6} from \cite{valero2017generalization}) to the PDEs \eqref{eqn:de_bruijn_pde}, it follows that
\begin{equation}
    \label{eqn:de_bruijn_identity}
    \begin{split}
        t \cdot \partial_{t} \calD_{f} \left( p_{t} \left\| q_{t} \right. \right) = \mathrm{Trace} \left\{ \calI_{f} \left( p_{t} \left\| q_{t} \right. \right) \right\}.
    \end{split}
\end{equation}
Hence, one can obtain the divergence derivative formulae \eqref{eqn:lemma:proof_thm:dimension_free_comparison_divergence_derivatives_v1} by applying the resulting de Bruijn's identity \eqref{eqn:de_bruijn_identity} with the following three generators:
\[
    \begin{split}
        f_{\mathrm{K}}: [0, +\infty) \to \bbR, \qquad f_{\mathrm{K}} (r) := \ &
        \begin{cases}
            r \log r - r + 1 & \textnormal{if } r > 0; \\
            1 & \textnormal{if } r = 0,
        \end{cases} \\
        f_{\mathrm{H}}: [0, +\infty) \to \bbR, \qquad f_{\mathrm{H}} (r) := \ & \left( \sqrt{r} - 1 \right)^2, \\
        f_{\mathrm{C}}: [0, +\infty) \to \bbR, \qquad f_{\mathrm{C}} (r) := \ & \left( r - 1 \right)^2.
    \end{split}
\]
In order to leverage \emph{Proposition 6} from \cite{valero2017generalization} to the PDE \eqref{eqn:de_bruijn_pde}, we need to provide the following analytic justifications for each choice of $f \in \left\{ f_{\mathrm{K}}, f_{\mathrm{H}}, f_{\mathrm{C}} \right\}$, which turns out to be rather involved:
\begin{enumerate} [label = (\roman*)]
    \item The local integrability of the functions $\partial_{t} \left\{ q_{t} f \left( \frac{p_{t}}{q_{t}} \right) \right\}$ and $\partial_{t}^2 \left\{ q_{t} f \left( \frac{p_{t}}{q_{t}} \right) \right\}$.
    \item The vanishing boundary conditions of the functions $\bsx q_{t} f \left( \frac{p_{t}}{q_{t}} \right)$ and $\nabla_{\bsx} \left\{ q_{t} f \left( \frac{p_{t}}{q_{t}} \right) \right\}$. 
\end{enumerate}

\newpage

\bibliographystyle{plain}
\bibliography{main.bib}

\newpage

\appendix

\section{Proof of Lemma \ref{lemma:proof_thm:dimension_free_comparison_divergence_derivatives}}
\label{sec:proof_lemma:proof_thm:dimension_free_comparison_divergence_derivatives}

\subsection{Two basic identities}
\label{subsec:basic_identities}

We first provide the two identities used to derive the divergence derivative formulas. The score identity represents spatial derivatives through posterior means, while integration by parts converts derivatives of the mixture densities into spatial derivatives of the test functions. Their proofs, together with the growth and integrability checks needed below, are deferred to Appendix \ref{subsec:analytic_justifications}.

\indent The first identity is the score identity for the posterior means of the original mixing vector \eqref{eqn:interpolation_standard_gaussian_gaussian_location_mixtures_v3}. Because $p_t$ and $q_t$ are the density functions of $\bsX (t) = t \bfTheta + \bsZ$, when $\bfTheta \sim \pi$ and $\bfTheta \sim \eta$, respectively, alongside with $\bsZ \sim \calN_{d} \left( \bfzero_d, \bsI_d \right)$ independent of $\bfTheta$, the following equations can be obtained by applying Tweedie's formula \cite{robbins1992empirical, efron2011tweedie}:
\begin{equation}
    \label{eqn:lemma:posterior_means_score_identities}
    \begin{split}
        \nabla_{\bsx} \log p_t (\bsx) = - \bsx + t \bsu_t (\bsx), \qquad \nabla_{\bsx} \log q_t (\bsx) = - \bsx + t \bsv_t (\bsx), \qquad \nabla_{\bsx} \log \frac{p_t (\bsx)}{q_t (\bsx)} =  t \bfDelta_t (\bsx).
    \end{split}
\end{equation}

\indent The second identity is the Gaussian integration by parts identity, established through the following lemma:

\begin{lemma} [Gaussian integration by parts]
\label{lemma:gaussian_integration_by_parts}
Let $g \in \calC^{1} \left( \bbR^d \right)$ and suppose that there exist some constants $L \in [0, +\infty)$ and $c \in [0, +\infty)$ such that
\begin{equation}
    \label{eqn:lemma:gaussian_integration_by_parts_v1}
    \begin{split}
        \left| g (\bsx) \right| + \left\| \nabla_{\bsx} g (\bsx) \right\|_{2} \leq L \exp \left( c \left\| \bsx \right\|_{2} \right), \qquad \bsx \in \bbR^d.
    \end{split}
\end{equation}
Then for every $t \in \bbR$, all the integrals below converge absolutely and it holds that
\begin{equation}
    \label{eqn:lemma:gaussian_integration_by_parts_v2}
    \begin{split}
        \int_{\bbR^d}  \partial_{t} p_t (\bsx) g (\bsx) \mathrm{d} \bsx = \ & \int_{\bbR^d} p_t (\bsx) \cdot \nabla_{\bsx}^{\top} g (\bsx) \bsu_{t} (\bsx) \mathrm{d} \bsx, \\
        \int_{\bbR^d} \partial_{t} q_t (\bsx) g (\bsx) \mathrm{d} \bsx = \ & \int_{\bbR^d} q_t (\bsx) \cdot \nabla_{\bsx}^{\top} g (\bsx) \bsv_{t} (\bsx) \mathrm{d} \bsx.
    \end{split}
\end{equation}
\end{lemma}

\subsection{Proof of the divergence derivative formulas}
\label{subsec:proof_divergence_derivative_formulas}

We now apply the two identities from Appendix \ref{subsec:basic_identities}. Finiteness and $\calC^1$-regularity, differentiation under the integral sign, and the exponential growth conditions required for integration by parts are justified in Appendix \ref{subsec:analytic_justifications}. Before embarking on the proofs, we recall the notion of the likelihood ratio $\rho_{t} := \frac{p_{t}}{q_{t}}$ for $t \in \bbR$.

\paragraph{Derivative of $J: \bbR \to [0, +\infty]$.}
To compute the derivative of $J: \bbR \to [0, +\infty]$, put
\[
    \begin{split}
        \calJ_{0} \left( t, \bsx \right) := \log \rho_{t} (\bsx) \cdot p_t (\bsx), \qquad \calJ_{1} \left( t, \bsx \right) := \log \rho_{t} (\bsx) \cdot q_t (\bsx)
    \end{split}
\]
for $\left( t, \bsx \right) \in \bbR \times \bbR^d$, and write $J: \bbR \to [0, +\infty]$ as
\begin{equation}
    \label{eqn:proof_subsec:proof_divergence_derivative_formulas_v1}
    \begin{split}
        J (t) = \int_{\bbR^d} \calJ_{0} \left( t, \bsx \right) \mathrm{d} \bsx - \int_{\bbR^d} \calJ_{1} \left( t, \bsx \right) \mathrm{d} \bsx, \qquad t \in \bbR.
    \end{split}
\end{equation}
Pointwise differentiation gives
\begin{equation}
    \label{eqn:proof_subsec:proof_divergence_derivative_formulas_v2}
    \begin{split}
        \partial_t \calJ_{0} \left( t, \bsx \right) = \ & \partial_{t} p_t (\bsx) \left\{ 1 + \log \rho_{t} (\bsx) \right\} - \rho_{t} (\bsx) \partial_{t} q_t (\bsx), \\
        \partial_t \calJ_{1} \left( t, \bsx \right) = \ & \partial_{t} q_t (\bsx) \left\{ \log \rho_{t} (\bsx) - 1 \right\} + \frac{\partial_{t} p_t (\bsx)}{\rho_{t} (\bsx)}.
    \end{split}
\end{equation}
We first evaluate the integral $\int_{\bbR^d} \partial_t \calJ_{0} \left( t, \bsx \right) \mathrm{d} \bsx$: since Lemma \ref{lemma:likelihood_ratio_exponential_growth_condition} verifies that $\log \rho_t : \bbR^d \to \bbR$ and $\rho_t : \bbR^d \to (0, +\infty)$ satisfy the exponential growth condition \eqref{eqn:lemma:gaussian_integration_by_parts_v1}, the Gaussian integration by parts identity \eqref{eqn:lemma:gaussian_integration_by_parts_v2} implies
\begin{equation}
    \label{eqn:proof_subsec:proof_divergence_derivative_formulas_v3}
    \begin{split}
        \int_{\bbR^d} \partial_t \calJ_{0} \left( t, \bsx \right) \mathrm{d} \bsx \stackrel{\textnormal{(a)}}{=} \ & \int_{\bbR^d} \partial_{t} p_t (\bsx) \log \rho_{t} (\bsx) \mathrm{d} \bsx - \int_{\bbR^d} \partial_{t} q_t (\bsx) \rho_{t} (\bsx) \mathrm{d} \bsx \\
        = \ & \int_{\bbR^d} p_t (\bsx) \underbrace{\nabla_{\bsx}^{\top} \log \rho_{t} (\bsx)}_{= \ t \bfDelta_{t}^{\top} (\bsx)} \bsu_{t} (\bsx) \mathrm{d} \bsx - \int_{\bbR^d} q_t (\bsx) \underbrace{\nabla_{\bsx}^{\top} \rho_{t} (\bsx)}_{= \ t \rho_{t} (\bsx) \bfDelta_{t}^{\top} (\bsx)} \bsv_{t} (\bsx) \mathrm{d} \bsx \\
        \stackrel{\textnormal{(b)}}{=} \ & t \int_{\bbR^d} p_t (\bsx) \bfDelta_{t}^{\top} (\bsx) \bsu_{t} (\bsx) \mathrm{d} \bsx - t \int_{\bbR^d} p_t (\bsx) \bfDelta_{t}^{\top} (\bsx) \bsv_{t} (\bsx) \mathrm{d} \bsx \\
        = \ & t \int_{\bbR^d} p_t (\bsx) \left\| \bfDelta_{t} (\bsx) \right\|_{2}^2 \mathrm{d} \bsx,
    \end{split}
\end{equation}
where the step (a) holds from $\int_{\bbR^d} \partial_{t} p_t (\bsx) \mathrm{d} \bsx = 0$ (the first piece of \eqref{eqn:proof_finiteness_regularity_divergence_curves_v1}) and the step (b) uses the third piece from \eqref{eqn:lemma:posterior_means_score_identities}. Next, we evaluate the integral $\int_{\bbR^d} \partial_t \calJ_{1} \left( t, \bsx \right) \mathrm{d} \bsx$: since Lemma \ref{lemma:likelihood_ratio_exponential_growth_condition} justifies the exponential growth conditions for $\log \rho_t : \bbR^d \to \bbR$ and $\frac{1}{\rho_t} : \bbR^d \to (0, +\infty)$, one can use the Gaussian integration by parts identity \eqref{eqn:lemma:gaussian_integration_by_parts_v2} to get
\begin{equation}
    \label{eqn:proof_subsec:proof_divergence_derivative_formulas_v4}
    \begin{split}
        \int_{\bbR^d} \partial_t \calJ_{1} \left( t, \bsx \right) \mathrm{d} \bsx \stackrel{\textnormal{(c)}}{=} \ & \int_{\bbR^d} \partial_{t} q_t (\bsx) \log \rho_{t} (\bsx) \mathrm{d} \bsx + \int_{\bbR^d} \frac{\partial_{t} p_t (\bsx)}{\rho_{t} (\bsx)} \mathrm{d} \bsx \\
        = \ & \int_{\bbR^d} q_t (\bsx) \underbrace{\nabla_{\bsx}^{\top} \log \rho_{t} (\bsx)}_{= \ t \bfDelta_{t}^{\top} (\bsx)} \bsv_{t} (\bsx) \mathrm{d} \bsx + \int_{\bbR^d} p_t (\bsx) \underbrace{\nabla_{\bsx}^{\top} \left( \frac{1}{\rho_{t}} \right) (\bsx)}_{= \ - \frac{t}{\rho_{t} (\bsx)} \bfDelta_{t}^{\top} (\bsx)} \bsu_{t} (\bsx) \mathrm{d} \bsx \\
        \stackrel{\textnormal{(d)}}{=} \ & t \int_{\bbR^d} q_t (\bsx) \bfDelta_{t}^{\top} (\bsx) \bsv_{t} (\bsx) \mathrm{d} \bsx - t \int_{\bbR^d} q_t (\bsx) \bfDelta_{t}^{\top} (\bsx) \bsu_{t} (\bsx) \mathrm{d} \bsx \\
        = \ & - t \int_{\bbR^d} q_t (\bsx) \left\| \bfDelta_{t} (\bsx) \right\|_{2}^2 \mathrm{d} \bsx,
    \end{split}
\end{equation}
where the step (c) arises from the fact  $\int_{\bbR^d} \partial_{t} q_t (\bsx) \mathrm{d} \bsx = 0$ (the second identity from \eqref{eqn:proof_finiteness_regularity_divergence_curves_v1}), and the step (d) follows due to the third identity from the equation \eqref{eqn:lemma:posterior_means_score_identities} together with the fact \eqref{eqn:lemma:likelihood_ratio_exponential_growth_condition_v1}. Taking two pieces \eqref{eqn:proof_subsec:proof_divergence_derivative_formulas_v3} and \eqref{eqn:proof_subsec:proof_divergence_derivative_formulas_v4} collectively together with Lemma \ref{lemma:dominated_differentiation} reveals that
\[
    \begin{split}
        J' (t) = \int_{\bbR^d} \partial_t \calJ_{0} \left( t, \bsx \right) \mathrm{d} \bsx - \int_{\bbR^d} \partial_t \calJ_{1} \left( t, \bsx \right) \mathrm{d} \bsx = t \int_{\bbR^d} \left\{ p_t (\bsx) + q_t (\bsx) \right\} \left\| \bfDelta_{t} (\bsx) \right\|_{2}^2 \mathrm{d} \bsx, \qquad t \in \bbR,
    \end{split}
\]
concluding the proof of the first identity from \eqref{eqn:lemma:proof_thm:dimension_free_comparison_divergence_derivatives_v1}.

\paragraph{Derivative of $C : \bbR \to [0, +\infty)$.}
We define $\calC : \bbR \times \bbR^d \to (0, +\infty)$ as $\calC \left( t, \bsx \right) := \frac{p_{t}^2 (\bsx)}{q_t (\bsx)}$ so that
\[
    \begin{split}
        C (t) = \int_{\bbR^d} \calC \left( t, \bsx \right) \mathrm{d} \bsx - 1, \qquad t \in \bbR.
    \end{split}
\]
Pointwise differentiation yields
\begin{equation}
    \label{eqn:proof_subsec:proof_divergence_derivative_formulas_v5}
    \begin{split}
        \partial_{t} \calC \left( t, \bsx \right) = \partial_{t} p_t (\bsx) \cdot \rho_{t} (\bsx) + p_t (\bsx) \cdot \partial_{t} \rho_{t} (\bsx) = 2 \partial_{t} p_t (\bsx) \cdot \rho_{t} (\bsx) - \partial_{t} q_t (\bsx) \cdot \rho_{t}^2 (\bsx).
    \end{split}
\end{equation}
Applying Lemma \ref{lemma:dominated_differentiation} and then the Gaussian integration by parts identity \eqref{eqn:lemma:gaussian_integration_by_parts_v2} gives
\[
    \begin{split}
        C' (t) 
        \stackrel{\textnormal{(e)}}{=} \ & 2 \int_{\bbR^d} 
        \partial_{t} p_t (\bsx) \cdot \rho_{t} (\bsx) \mathrm{d} \bsx - \int_{\bbR^d} 
        \partial_{t} q_t (\bsx) \cdot \rho_{t}^2 (\bsx) \mathrm{d} \bsx \\
        \stackrel{\textnormal{(f)}}{=} \ & 2 \int_{\bbR^d} 
        p_t (\bsx) \underbrace{\nabla_{\bsx}^{\top} \rho_{t} (\bsx)}_{= \ t \rho_{t} (\bsx) \bfDelta_{t}^{\top} (\bsx)} \bsu_{t} (\bsx) \mathrm{d} \bsx - \int_{\bbR^d} 
        q_t (\bsx) \underbrace{\nabla_{\bsx}^{\top} \left( \rho_{t}^2 \right) (\bsx)}_{= \ 2 t \rho_{t}^2 (\bsx) \bfDelta_{t}^{\top} (\bsx)} \bsv_{t} (\bsx) \mathrm{d} \bsx \\
        \stackrel{\textnormal{(g)}}{=} \ & 2 t \int_{\bbR^d} \calC \left( t, \bsx \right) \bfDelta_{t}^{\top} (\bsx) \bsu_{t} (\bsx) \mathrm{d} \bsx - 2 t \int_{\bbR^d} \calC \left( t, \bsx \right) \bfDelta_{t}^{\top} (\bsx) \bsv_{t} (\bsx) \mathrm{d} \bsx \\
        = \ & 2 t \int_{\bbR^d} \calC \left( t, \bsx \right) \left\| \bfDelta_{t} (\bsx)  \right\|_{2}^2 \mathrm{d} \bsx,
    \end{split}
\]
where the step (e) holds from the identity \eqref{eqn:proof_subsec:proof_divergence_derivative_formulas_v5}, the step (f) uses the Gaussian integration by parts identity \eqref{eqn:lemma:gaussian_integration_by_parts_v2}, and the step (g) is a result of \eqref{eqn:lemma:likelihood_ratio_exponential_growth_condition_v1}. The application of the Gaussian integration by parts \eqref{eqn:lemma:gaussian_integration_by_parts_v2} is justified via the exponential growth conditions for $\rho_t$ and $\rho_t^2$ from Lemma \ref{lemma:likelihood_ratio_exponential_growth_condition}. This establishes the formula for the derivative of the $\chi^2$-divergence in \eqref{eqn:lemma:proof_thm:dimension_free_comparison_divergence_derivatives_v1}.

\paragraph{Derivative of the Hellinger affinity and $D: \bbR \to [0, 2]$.}
Define the Hellinger affinity by
\begin{equation}
    \label{eqn:proof_subsec:proof_divergence_derivative_formulas_v6}
    \begin{split}
        A (t) := \int_{\bbR^d} \sqrt{p_{t} (\bsx) q_{t} (\bsx)} \mathrm{d} \bsx = 1 - \frac{D (t)}{2}, \qquad t \in \bbR.
    \end{split}
\end{equation}
Put $\calA \left( t, \bsx \right) := \sqrt{p_t (\bsx) q_t (\bsx)}$ for $\left( t, \bsx \right) \in \bbR \times \bbR^d$. Pointwise differentiation then gives
\begin{equation}
    \label{eqn:proof_subsec:proof_divergence_derivative_formulas_v7}
    \begin{split}
        \partial_{t} \calA \left( t, \bsx \right) = \frac{1}{2 \sqrt{\rho_{t} (\bsx)}} \partial_{t} p_t (\bsx) + \frac{\sqrt{\rho_{t} (\bsx)}}{2} \partial_{t} q_t (\bsx),
    \end{split}
\end{equation}
Differentiating under the integral sign (Lemma \ref{lemma:dominated_differentiation}) and integrating by parts, we obtain
\[
    \begin{split}
        A' (t) \stackrel{\textnormal{(h)}}{=} \ & \frac{1}{2} \int_{\bbR^d} \partial_{t} p_t (\bsx) \cdot \frac{1}{\sqrt{\rho_t (\bsx)}} \mathrm{d} \bsx + \frac{1}{2} \int_{\bbR^d} \partial_{t} q_t (\bsx) \sqrt{\rho_t (\bsx)} \mathrm{d} \bsx \\
        \stackrel{\textnormal{(i)}}{=} \ & \frac{1}{2} \int_{\bbR^d} p_t (\bsx) \underbrace{\nabla_{\bsx}^{\top} \left( \frac{1}{\sqrt{\rho_t}} \right) (\bsx)}_{= \ - \frac{t}{2 \sqrt{\rho_{t} (\bsx)}} \bfDelta_{t} (\bsx)} \bsu_{t} (\bsx) \mathrm{d} \bsx + \frac{1}{2} \int_{\bbR^d} q_t (\bsx) \underbrace{\nabla_{\bsx}^{\top} \sqrt{\rho_t} (\bsx)}_{= \ \frac{t}{2} \sqrt{\rho_{t} (\bsx)} \bfDelta_{t} (\bsx)} \bsv_{t} (\bsx) \mathrm{d} \bsx \\
        \stackrel{\textnormal{(j)}}{=} \ & - \frac{t}{4} \int_{\bbR^d} \sqrt{p_t (\bsx) q_t (\bsx)} \left\| \bfDelta_{t} (\bsx) \right\|_{2}^2 \mathrm{d} \bsx,
    \end{split}
\]
where the step (h) holds by the identity \eqref{eqn:proof_subsec:proof_divergence_derivative_formulas_v7}, the step (i) makes use of the Gaussian integration by parts identity \eqref{eqn:lemma:gaussian_integration_by_parts_v2}, and the step (j) comes from \eqref{eqn:lemma:likelihood_ratio_exponential_growth_condition_v1}. Here, application of the Gaussian integration by parts \eqref{eqn:lemma:gaussian_integration_by_parts_v2} is justified by the exponential growth conditions for $\sqrt{\rho_{t}}$ and $\frac{1}{\sqrt{\rho_{t}}}$, which follow from Lemma \ref{lemma:likelihood_ratio_exponential_growth_condition}.

\indent Finally, the relation $D = 2 - 2A$ implies
\[
    \begin{split}
        D' (t) = - 2 A' (t) = \frac{t}{2} \int_{\bbR^d} \sqrt{p_t (\bsx) q_t (\bsx)} \left\| \bfDelta_t (\bsx) \right\|_{2}^2 \mathrm{d} \bsx, 
    \end{split}
\]
This gives the formula for the derivative of the squared Hellinger distance in \eqref{eqn:lemma:proof_thm:dimension_free_comparison_divergence_derivatives_v1}. All three derivative calculations are now complete; we next elaborate on analytic justifications required for the calculations above.

\subsection{Analytic justifications}
\label{subsec:analytic_justifications}

We now supply the analytic details for Appendix \ref{subsec:proof_divergence_derivative_formulas}. We first establish the Gaussian domination bounds, dominated differentiation, and density envelopes bounds. These results justify the two supporting lemmas stated in Appendix \ref{subsec:basic_identities}. We then verify the likelihood-ratio growth conditions, and the finiteness and $\calC^1$-regularity for the divergence curves.

\subsubsection{Gaussian domination}
\label{subsubsec:gaussian_domination}

\begin{lemma} [Gaussian domination]
\label{lemma:gaussian_domination}
It holds for every $\left( c, k \right) \in [0, +\infty) \times \bbN_0$ that
\begin{equation}
    \label{eqn:lemma:gaussian_domination_v1}
    \begin{split}
        \int_{\bbR^d} \phi_d (\bsx) \left( 1 + \left\| \bsx \right\|_{2} \right)^k \exp \left( c \left\| \bsx \right\|_{2} \right) \mathrm{d} \bsx \leq 2^{k + \frac{d}{2}} k! \exp \left\{ (c+1)^2 \right\}.
    \end{split}
\end{equation}
In particular, for $k = 0$, the following sharper bound holds:
\begin{equation}
    \label{eqn:lemma:gaussian_domination_v2}
    \begin{split}
        \int_{\bbR^d} \phi_d (\bsx) \exp \left( c \left\| \bsx \right\|_{2} \right) \mathrm{d} \bsx \leq 2^{\frac{d}{2}} \exp \left( c^2 \right).
    \end{split}
\end{equation}
\end{lemma}

\begin{proof} [Proof of Lemma \ref{lemma:gaussian_domination}]
To begin with, one can easily observe from the exponential series that
\begin{equation}
    \label{eqn:proof_lemma:gaussian_domination_v1}
    \begin{split}
        r^j \leq j! \exp (r), \qquad (r, j) \in [0, +\infty) \times \bbN_0.
    \end{split}
\end{equation}
Leveraging the basic inequality \eqref{eqn:proof_lemma:gaussian_domination_v1}, one has
\begin{equation}
    \label{eqn:proof_lemma:gaussian_domination_v2}
    \begin{split}
        (1 + r)^k = \sum_{j=0}^{k} \binom{k}{j} r^j \leq \exp (r) \sum_{j=0}^{k} \binom{k}{j} j! \leq k! \exp (r) \underbrace{\sum_{j=0}^{k} \binom{k}{j}}_{= \ 2^k} = 2^k k! \exp (r), \qquad (r, k) \in [0, +\infty) \times \bbN_0. 
    \end{split}
\end{equation}
Therefore, it follows that
\[
    \begin{split}
        \int_{\bbR^d} \phi_d (\bsx) \left( 1 + \left\| \bsx \right\|_{2} \right)^k \exp \left( c \left\| \bsx \right\|_{2} \right) \mathrm{d} \bsx \leq \ & 2^k k! \left( 2 \pi \right)^{- \frac{d}{2}} \int_{\bbR^d} \exp \left\{ - \frac{1}{2} \left\| \bsx \right\|_{2}^2 + (c+1) \left\| \bsx \right\|_{2} \right\} \mathrm{d} \bsx \\
        \stackrel{\textnormal{(a)}}{\leq} \ & 2^k k! \left( 2 \pi \right)^{- \frac{d}{2}} \exp \left\{ (c+1)^2 \right\} \underbrace{\int_{\bbR^d} \exp \left( - \frac{1}{4} \left\| \bsx \right\|_{2}^2 \right) \mathrm{d} \bsx}_{= \ \left( 4 \pi \right)^{\frac{d}{2}}} \\
        = \ & 2^{k + \frac{d}{2}} k! \exp \left\{ (c+1)^2 \right\},
    \end{split}
\]
which establishes the first desired inequality \eqref{eqn:lemma:gaussian_domination_v1}, where the step (a) makes use of the following simple bound: for any $r \in [0, +\infty)$, we have
\[
    \begin{split}
        - \frac{r^2}{2} + (c+1) r = - \frac{r^2}{4} - \left\{ \frac{r}{2} - (c+1) \right\}^2 + (c+1)^2 \leq - \frac{r^2}{4} + (c+1)^2.
    \end{split}
\]
For the second desired inequality \eqref{eqn:lemma:gaussian_domination_v2}, we find
\[
    \begin{split}
        \int_{\bbR^d} \phi_d (\bsx) \exp \left( c \left\| \bsx \right\|_{2} \right) \mathrm{d} \bsx = \ & \left( 2 \pi \right)^{- \frac{d}{2}} \int_{\bbR^d} \exp \left( - \frac{1}{2} \left\| \bsx \right\|_{2}^2 + c \left\| \bsx \right\|_{2} \right) \mathrm{d} \bsx \\
        \stackrel{\textnormal{(b)}}{\leq} \ & \left( 2 \pi \right)^{- \frac{d}{2}} \exp \left( c^2 \right) \underbrace{\int_{\bbR^d} \exp \left( - \frac{1}{4} \left\| \bsx \right\|_{2}^2 \right) \mathrm{d} \bsx}_{= \ \left( 4 \pi \right)^{\frac{d}{2}}} \\
        = \ & 2^{\frac{d}{2}} \exp \left( c^2 \right),
    \end{split}
\]
where the step (b) holds due to the following simple inequality: for any $r \in [0, +\infty)$, one has
\[
    \begin{split}
        - \frac{r^2}{2} + c r = - \frac{r^2}{4} - \left( \frac{r}{2} - c \right)^2 + c^2 \leq - \frac{r^2}{4} + c^2.
    \end{split}
\]
    
\end{proof}

\noindent In order to abbreviate explicit bounds from Lemma \ref{lemma:gaussian_domination}, we define
\[ 
    \begin{split}
        \calK \left( c; d, k \right) := 2^{k + \frac{d}{2}} k! \exp \left\{ (c+1)^2 \right\}, \qquad \calS \left( c; d \right) := 2^{\frac{d}{2}} \exp \left( c^2 \right).
    \end{split}
\]

\subsubsection{Dominated differentiation}
\label{subsubsec:dominated_differentiation}

Let us next present a useful result that provides sufficient conditions for when one can freely interchange the order of differentiation and integration.

\begin{lemma}
\label{lemma:dominated_differentiation}
Let $\nu : \calX \to [0, +\infty]$ be a $\sigma$-finite measure defined on a measurable space $\left( \bbX, \calX \right)$, $I \subseteq \bbR$ be a non-empty open interval, and $H : I \times \bbX \to \bbR$ be any $\left( \calB (I) \otimes \calX \right)$-measurable function. Suppose that there is a measurable subset $N \in \calX$ with $\nu (N) = 0$ satisfying the following properties:
\begin{enumerate} [label = (\roman*)]
    \item $H \left( \cdot, x \right) \in \calC^1 (I)$ for every $x \in \bbX \setminus N$.
    \item For each compact interval $K \subseteq I$, there exist measurable functions $P_K : \bbX \to [0, +\infty)$ and $Q_K : \bbX \to [0, +\infty)$ such that
    \begin{equation}
        \label{eqn:lemma:dominated_differentiation_v1}
        \begin{split}
            \int_{\bbX} \left\{ P_K (x) + Q_K (x) \right\} \nu \left( \mathrm{d} x \right) < +\infty
        \end{split}
    \end{equation}
    and
    \begin{equation}
        \label{eqn:lemma:dominated_differentiation_v2}
        \begin{split}
            \left| H (t, x) \right| \leq P_K (x), \qquad \left| \partial_t H (t, x) \right| \leq Q_K (x), \qquad (t, x) \in K \times \left( \bbX \setminus N \right).
        \end{split}
    \end{equation}
\end{enumerate}
Then, the function $F : I \to \bbR$, defined as
\[
    \begin{split}
        F (t) := \int_{\bbX} H (t, x) \nu \left( \mathrm{d} x \right), \qquad t \in I,
    \end{split}
\]
belongs to the class $\calC^1 (I)$ and its derivative is given by
\begin{equation}
    \label{eqn:lemma:dominated_differentiation_v3}
    \begin{split}
            F' (t) = \int_{\bbX} \partial_t H (t, x) \nu \left( \mathrm{d} x \right), \qquad t \in I.
    \end{split}
\end{equation}
\end{lemma}

\begin{proof} [Proof of Lemma \ref{lemma:dominated_differentiation}]
To begin with, fix any $t \in I$ and choose a sufficiently small $\delta \in (0, +\infty)$ so that $\left[ t - \delta, t + \delta \right] \subseteq I$. Set $K = \left[ t - \delta, t + \delta \right] \subseteq I$. For every $0 < |h| \leq \delta$, the mean-value theorem gives
\[  
    \begin{split}
        \left| \frac{H \left( t+h, x \right) - H (t, x)}{h} \right| \leq Q_K (x), \qquad x \in \bbX \setminus N.
    \end{split}
\]
Applying the dominated convergence theorem, one can conclude that
\[
    \begin{split}
        F' (t) = \ & \lim_{h \to 0} \int_{\bbX} \frac{H \left( t+h, x \right) - H (t, x)}{h} \nu \left( \mathrm{d} x \right) \\
        = \ & \lim_{h \to 0} \int_{\bbX} \frac{H \left( t+h, x \right) - H (t, x)}{h} \cdot \mathbbm{1}_{\bbX \setminus N} (x) \nu \left( \mathrm{d} x \right) \\
        = \ & \int_{\bbX} \underbrace{\lim_{h \to 0} \frac{H \left( t+h, x \right) - H (t, x)}{h} \cdot \mathbbm{1}_{\bbX \setminus N} (x)}_{= \ \partial_t H (t, x) \cdot \mathbbm{1}_{\bbX \setminus N} (x)} \nu \left( \mathrm{d} x \right) \\
        = \ & \int_{\bbX} \partial_t H (t, x) \nu \left( \mathrm{d} x \right),
    \end{split}
\]
which establishes the desired identity \eqref{eqn:lemma:dominated_differentiation_v3}. So, it only remains to prove the continuity of the derivative $F' : I \to \bbR$. Towards this end, we consider any sequence $\left\{ t_n : n \in \bbN \right\}$ in $I$ that converges to $t \in I$. Then, there exists some $N \in \bbN$ such that $t_n \in K$ for every $n \geq N$. 
Since the assumption \eqref{eqn:lemma:dominated_differentiation_v2} implies
\[
    \begin{split}
        \left| \partial_t H (t, x) \right| \cdot \mathbbm{1}_{\bbX \setminus N} (x) \leq Q_K (x), \qquad (t, x) \in K \times \bbX,
    \end{split}
\]
the dominated convergence theorem yields
\[
    \begin{split}
        \lim_{n \to \infty} F' \left( t_n \right) = \lim_{n \to \infty} \int_{\bbX} \partial_t H \left( t_n, x \right) \cdot \mathbbm{1}_{\bbX \setminus N} (x) \nu \left( \mathrm{d} x \right) \stackrel{\textnormal{(a)}}{=} \int_{\bbX} \partial_t H \left( t, x \right) \cdot \mathbbm{1}_{\bbX \setminus N} (x) \nu \left( \mathrm{d} x \right) = F' (t),
    \end{split}
\]
where the step (a) makes use of the assumption that $H \left( \cdot, x \right) \in \calC^1 (I)$ for every $x \in \bbX \setminus N$. This concludes the proof of the desired conclusion that $F \in \calC^1 (I)$.
    
\end{proof}

\subsubsection{Density envelopes and differentiation}
\label{subsubsec:density_envelopes_differentiation}

\begin{lemma} [Differentiation with respect to the interpolation parameter]
\label{lemma:interpolation_parameter_differentiation_density_envelopes}
Fix any $T \in (0, +\infty)$ and define a function $a_T : \bbR^d \to [0, +\infty)$ as
\begin{equation}
    \label{eqn:lemma:interpolation_parameter_differentiation_density_envelopes_v1}
    \begin{split}
        a_T (\bsx) := M \left\| \bsx \right\|_2 + T M^2.
    \end{split}
\end{equation}
Then, the following properties hold:
\begin{enumerate} [label = (\roman*)]
    \item It holds for any $\left( t, \bsx \right) \in [-T, T] \times \bbR^d$ that
    \begin{equation}
        \label{eqn:lemma:interpolation_parameter_differentiation_density_envelopes_v2}
        \begin{split}
            \phi_d (\bsx) \exp \left( - TM \left\| \bsx \right\|_2 - \frac{T^2 M^2}{2} \right) \leq \ & p_{t} (\bsx) \leq \phi_d (\bsx) \exp \left( TM \left\| \bsx \right\|_2 \right), \\
            \phi_d (\bsx) \exp \left( - TM \left\| \bsx \right\|_2 - \frac{T^2 M^2}{2} \right) \leq \ & q_{t} (\bsx) \leq \phi_d (\bsx) \exp \left( TM \left\| \bsx \right\|_2 \right).
        \end{split}
    \end{equation}
    \item The derivative with respect to the interpolation parameter is given by
    \begin{equation}
        \label{eqn:lemma:interpolation_parameter_differentiation_density_envelopes_v3}
        \begin{split}
            \partial_{t} p_{t} (\bsx) = \ & \int_{\bbB_{2}^d (M)} \left( \bftheta^{\top} \bsx - t \left\| \bftheta \right\|_{2}^2 \right) \phi_d \left( \bsx - t \bftheta \right) \pi \left( \mathrm{d} \bftheta \right), \\
            \partial_{t} q_{t} (\bsx) = \ & \int_{\bbB_{2}^d (M)} \left( \bftheta^{\top} \bsx - t \left\| \bftheta \right\|_{2}^2 \right) \phi_d \left( \bsx - t \bftheta \right) \eta \left( \mathrm{d} \bftheta \right).
        \end{split}
    \end{equation}
    In particular, it follows for every $\left( t, \bsx \right) \in [-T, T] \times \bbR^d$ that
    \begin{equation}
        \label{eqn:lemma:interpolation_parameter_differentiation_density_envelopes_v4}
        \begin{split}
            \left| \partial_{t} p_t (\bsx) \right| \leq a_{T} (\bsx) p_t (\bsx), \qquad \left| \partial_{t} q_t (\bsx) \right| \leq a_{T} (\bsx) q_t (\bsx).
        \end{split}
    \end{equation}
    \item The spatial derivatives of $p_{t}$ and $q_{t}$ are given by
    \begin{equation}
        \label{eqn:lemma:interpolation_parameter_differentiation_density_envelopes_v5}
        \begin{split}
            \nabla_{\bsx} p_t (\bsx) = \ & -\int_{\bbB_{2}^d (M)} \left( \bsx - t \bftheta \right) \phi_d \left( \bsx - t \bftheta \right) \pi \left( \mathrm{d} \bftheta \right), \\
            \nabla_{\bsx} q_t (\bsx) = \ & -\int_{\bbB_{2}^d (M)} \left( \bsx - t \bftheta \right) \phi_d \left( \bsx - t \bftheta \right) \eta \left( \mathrm{d} \bftheta \right)
        \end{split}
    \end{equation}
    \item The functions $\left( t, \bsx \right) \in \bbR \times \bbR^d \mapsto p_{t} (\bsx) \in [0, +\infty)$ and $\left( t, \bsx \right) \in \bbR \times \bbR^d \mapsto q_{t} (\bsx) \in [0, +\infty)$ are jointly $\calC^1$.
\end{enumerate}
\end{lemma}

\begin{proof} [Proof of Lemma \ref{lemma:interpolation_parameter_differentiation_density_envelopes}]
Here, we only establish the results from Lemma \ref{lemma:interpolation_parameter_differentiation_density_envelopes} for $\pi$; the analogous results for $\eta$ immediately follows by using exactly the same arguments.
\medskip

\indent (\romannumeral 1) One can easily observe the following Gaussian shift identity:
\begin{equation}
    \label{eqn:proof_lemma:interpolation_parameter_differentiation_density_envelopes_v1}
    \begin{split}
        \phi_d \left( \bsx - t \bftheta \right) = \phi_d (\bsx) \exp \left( t \bftheta^{\top} \bsx - \frac{t^2 \left\| \bftheta \right\|_{2}^2}{2} \right), \qquad \left( t, \bsx, \bftheta \right) \in \bbR \times \bbR^d \times \bbR^d.
    \end{split}
\end{equation}
For any $\left( t, \bftheta \right) \in [-T, T] \times \bbB_{2}^d (M)$, we have
\begin{equation}
    \label{eqn:proof_lemma:interpolation_parameter_differentiation_density_envelopes_v2}
    \begin{split}
        - TM \left\| \bsx \right\|_{2} - \frac{T^2 M^2}{2} \leq t \bftheta^{\top} \bsx - \frac{t^2 \left\| \bftheta \right\|_{2}^2}{2} \leq TM \left\| \bsx \right\|_{2}, \qquad \bsx \in \bbR^d.
    \end{split}
\end{equation}
Combining two pieces \eqref{eqn:proof_lemma:interpolation_parameter_differentiation_density_envelopes_v1} and \eqref{eqn:proof_lemma:interpolation_parameter_differentiation_density_envelopes_v2} and then integrating over $\bftheta \in \bbB_{2}^d (M)$ with respect to $\pi$ finishes the proof of the part (\romannumeral 1).
\medskip

\indent (\romannumeral 2) For each fixed $\bsx \in \bbR^d$, let $V_{\bsx} : \bbR \times \bbB_{2}^d (M) \to \bbR$ be defined by
\[
    \begin{split}
        V_{\bsx} \left( t, \bftheta \right) := \phi_d \left( \bsx - t \bftheta \right).
    \end{split}
\]
We now choose any compact interval $K \subseteq \bbR$ and any sufficiently large $A \in (0, +\infty)$ such that $K \subseteq \left[ -A, A \right]$. Then, one can observe that
\[
    \begin{split}
        \left| V_{\bsx} \left( t, \bftheta \right) \right| \leq \left( 2 \pi \right)^{- \frac{d}{2}}, \qquad \left( t, \bftheta \right) \in K \times \bbB_{2}^d (M), 
    \end{split}
\]
and
\[
    \begin{split}
        \left| \partial_t V_{\bsx} \left( t, \bftheta \right) \right| \leq \left( 2 \pi \right)^{- \frac{d}{2}} \left( M \left\| \bsx \right\|_{2} + A M^2 \right), \qquad \left( t, \bftheta \right) \in K \times \bbB_{2}^d (M).
    \end{split}
\]
Here, note that the dominating integrands $\left( 2 \pi \right)^{- \frac{d}{2}}$ and $\left( 2 \pi \right)^{- \frac{d}{2}} \left( M \left\| \bsx \right\|_{2} + A M^2 \right)$ are both constants of $\bftheta \in \bbB_{2}^d (M)$, thereby $\pi$-integrable. Therefore, we may apply Lemma \ref{lemma:dominated_differentiation} to obtain
\[
    \begin{split}
        \partial_{t} p_{t} (\bsx) = \partial_t \int_{\bbB_{2}^d (M)} V_{\bsx} \left( t, \bftheta \right) \pi \left( \mathrm{d} \bftheta \right) = \int_{\bbB_{2}^d (M)} \partial_t V_{\bsx} \left( t, \bftheta \right) \pi \left( \mathrm{d} \bftheta \right) = \int_{\bbB_{2}^d (M)} \left( \bftheta^{\top} \bsx - t \left\| \bftheta \right\|_{2}^2 \right) \phi_d \left( \bsx - t \bftheta \right) \pi \left( \mathrm{d} \bftheta \right),
    \end{split}
\]
as desired in \eqref{eqn:lemma:interpolation_parameter_differentiation_density_envelopes_v3}. Also, Lemma \ref{lemma:dominated_differentiation} implies that $t \in \bbR \mapsto p_{t} (\bsx) \in [0, +\infty)$ belongs to the class $\calC^1 (\bbR)$.
\medskip

\indent (\romannumeral 3) Similar to the part (\romannumeral 2), for each fixed $\left( t, i \right) \in \bbR \times [d]$ and $\bsx_{-i} = \left( x_j : j \in [d] \setminus \{ i \} \right) \in \bbR^{[d] \setminus \{ i \}}$, let us consider the function $W_{t}^{(i)} : \bbR \times \bbB_{2}^d (M) \to \bbR$ defined as
\[
    \begin{split}
        W_{\left( t, \bsx_{-i} \right)}^{(i)}  \left( x_i, \bftheta \right) := \phi_d \left( \bsx - t \bftheta \right).
    \end{split}
\]
Next, we choose any compact interval $K' \subseteq \bbR$ and any sufficiently large $B \in (0, +\infty)$ so that $K' \subseteq \left[ -B, B \right]$. Then, we have
\[
    \begin{split}
        \left| W_{\left( t, \bsx_{-i} \right)}^{(i)} \left( x_i, \bftheta \right) \right| \leq \left( 2 \pi \right)^{- \frac{d}{2}}, \qquad \left( x_i, \bftheta \right) \in K' \times \bbB_{2}^d (M), 
    \end{split}
\]
and
\[
    \begin{split}
        \left| \partial_{x_i} W_{\left( t, \bsx_{-i} \right)}^{(i)} \left( x_i, \bftheta \right) \right| \leq \left( 2 \pi \right)^{- \frac{d}{2}} \left( B + |t| M \right), \qquad \left( x_i, \bftheta \right) \in K' \times \bbB_{2}^d (M). 
    \end{split}
\]
Noticing that the dominating integrands $\left( 2 \pi \right)^{- \frac{d}{2}}$ and $\left( 2 \pi \right)^{- \frac{d}{2}} \left( B + |t| M \right)$ are both constants of $\bftheta \in \bbB_{2}^d (M)$, one can leverage Lemma \ref{lemma:dominated_differentiation} to reveal that
\begin{equation}
    \label{eqn:proof_lemma:interpolation_parameter_differentiation_density_envelopes_v3}
    \begin{split}
        \partial_{x_i} p_{t} (\bsx) = \ & \partial_{x_i} \int_{\bbB_{2}^d (M)} W_{\left( t, \bsx_{-i} \right)}^{(i)} \left( x_i, \bftheta \right) \pi \left( \mathrm{d} \bftheta \right) \\
        = \ & \int_{\bbB_{2}^d (M)} \partial_{x_i} W_{\left( t, \bsx_{-i} \right)}^{(i)} \left( x_i, \bftheta \right) \pi \left( \mathrm{d} \bftheta \right) \\
        = \ & - \int_{\bbB_{2}^d (M)} \left( x_i - t \theta_i \right) \phi_d \left( \bsx - t \bftheta \right) \pi \left( \mathrm{d} \bftheta \right).
    \end{split}
\end{equation}
By aggregating the result \eqref{eqn:proof_lemma:interpolation_parameter_differentiation_density_envelopes_v3} over $i \in [d]$, we obtain the desired identity \eqref{eqn:lemma:interpolation_parameter_differentiation_density_envelopes_v5} and conclude that the function $\bsx \in \bbR^d \mapsto p_{t} (\bsx) \in [0, +\infty)$ belongs to the class $\calC^1 \left( \bbR^d \right)$.
\medskip

\indent (\romannumeral 4) Employing the dominated convergence theorem, one can easily reveal that the functions 
\[
    \begin{split}
        \left( t, \bsx \right) \in \bbR \times \bbR^d \qquad \mapsto & \qquad \partial_{t} p_{t} (\bsx) \in \bbR, \\
        \left( t, \bsx \right) \in \bbR \times \bbR^d \qquad \mapsto & \qquad \nabla_{\bsx} p_{t} (\bsx) \in \bbR^d
    \end{split}
\]
are all continuous. Hence, the function $\left( t, \bsx \right) \in \bbR \times \bbR^d \mapsto p_{t} (\bsx) \in [0, +\infty)$ belongs to the class $\calC^1 \left( \bbR \times \bbR^d \right)$.
\end{proof}

\indent We are now ready to present the proof for the Gaussian integration by parts identity stated in Lemma \ref{lemma:gaussian_integration_by_parts}.

\subsubsection{Proof of Lemma \ref{lemma:gaussian_integration_by_parts}}
\label{subsec:proof_lemma:gaussian_integration_by_parts}

\indent Here, we only prove the results of Lemma \ref{lemma:gaussian_integration_by_parts} for $\pi$; the analogous results for $\eta$ immediately follows by using exactly the same arguments. We fix any $t \in \bbR$ and choose $T \in (0, +\infty)$ so that $T > |t|$. For each fixed $\bftheta \in \bbB_{2}^{d} (M)$, put
\[
    \begin{split}
        \kappa \left( \bsx; t, \bftheta \right) := \phi_d \left( \bsx - t \bftheta \right).
    \end{split}
\]
One can easily observe that
\begin{equation}
    \label{eqn:proof_lemma:gaussian_integration_by_parts_v1}
    \begin{split}
        \partial_t \kappa \left( \bsx; t, \bftheta \right) = \ & \left( \bftheta^{\top} \bsx - t \left\| \bftheta \right\|_{2}^2 \right) \kappa \left( \bsx; t, \bftheta \right), \\
        \nabla_{\bsx} \kappa \left( \bsx; t, \bftheta \right) = \ & - \left( \bsx - t \bftheta \right) \kappa \left( \bsx; t, \bftheta \right),
    \end{split}
\end{equation}
which produces
\begin{equation}
    \label{eqn:proof_lemma:gaussian_integration_by_parts_v2}
    \begin{split}
        \partial_t \kappa \left( \bsx; t, \bftheta \right) = - \bftheta^{\top} \nabla_{\bsx} \kappa \left( \bsx; t, \bftheta \right).
    \end{split}
\end{equation}
Also, recall from the equations \eqref{eqn:proof_lemma:interpolation_parameter_differentiation_density_envelopes_v1} and \eqref{eqn:proof_lemma:interpolation_parameter_differentiation_density_envelopes_v2} that
\begin{equation}
    \label{eqn:proof_lemma:gaussian_integration_by_parts_v3}
    \begin{split}
        \phi_d (\bsx) \exp \left( - TM \left\| \bsx \right\|_{2} - \frac{T^2 M^2}{2} \right) \leq \kappa \left( \bsx; t, \bftheta \right) = \phi_d \left( \bsx - t \bftheta \right) \leq \phi_d (\bsx) \exp \left( TM \left\| \bsx \right\|_{2} \right)
    \end{split}
\end{equation}
for every $\left( t, \bsx, \bftheta \right) \in [-T, T] \times \bbR^d \times \bbB_{2}^{d} (M)$. Equipped with the above preliminaries, let us first establish the absolute integrability of the functions
\begin{equation}
    \label{eqn:proof_lemma:gaussian_integration_by_parts_v4}
    \begin{split}
        \bsx \in \bbR^d \qquad \mapsto & \qquad \partial_{t} f_{t}^{G} (\bsx) g (\bsx) \in \bbR, \\
        \bsx \in \bbR^d \qquad \mapsto & \qquad p_{t} (\bsx) \cdot \nabla_{\bsx}^{\top} g (\bsx) \bbE_{G} \bsu_{t} (\bsx) \in \bbR.
    \end{split}
\end{equation}
To this end, we obtain the following bound by combining the inequality \eqref{eqn:proof_lemma:gaussian_integration_by_parts_v3} together with the assumption \eqref{eqn:lemma:gaussian_integration_by_parts_v1} on the function $g \in \calC^1 \left( \bbR^d \right)$: for each fixed $\bftheta \in \bbB_{2}^{d} (M)$,
\begin{equation}
    \label{eqn:proof_lemma:gaussian_integration_by_parts_v5}
    \begin{split}
        \kappa \left( \bsx; t, \bftheta \right) \left| g (\bsx) \right| \leq L \phi_d (\bsx) \exp \left\{ \left( c + TM \right) \left\| \bsx \right\|_{2} \right\}, \qquad \bsx \in \bbR^d.
    \end{split}
\end{equation}
Thus, by using the derivative of $\kappa \left( \bsx; t, \bftheta \right)$ with respect to the interpolation parameter from \eqref{eqn:proof_lemma:gaussian_integration_by_parts_v1}, we have
\[
    \begin{split}
        \left| \partial_{t} \kappa \left( \bsx; t, \bftheta \right) g (\bsx) \right| = \ & \left| \bftheta^{\top} \bsx - t \left\| \bftheta \right\|_{2}^2 \right| \cdot \kappa \left( \bsx; t, \bftheta \right) \left| g (\bsx) \right| \\
        \stackrel{\textnormal{(a)}}{\leq} \ & \left( M \left\| \bsx \right\|_{2} + TM^2 \right) \cdot L \phi_d (\bsx) \exp \left\{ \left( c + TM \right) \left\| \bsx \right\|_{2} \right\} \\
        \leq \ & TM^2 L \cdot \phi_d (\bsx) \exp \left\{ \left( c + TM \right) \left\| \bsx \right\|_{2} \right\} + ML \phi_d (\bsx) 
        \left( 1 + \left\| \bsx \right\|_{2} \right) \exp \left\{ \left( c + TM \right) \left\| \bsx \right\|_{2} \right\},
    \end{split}
\]
where the step (a) follows since $|t| < T$ and $\left\| \bftheta \right\|_{2} \leq M$ together with the Cauchy-Schwarz inequality. This yields
\begin{equation}
    \label{eqn:proof_lemma:gaussian_integration_by_parts_v6}
    \begin{split}
        \left| \partial_{t} p_{t} (\bsx) (\bsx) g (\bsx) \right| = \ & \left| \partial_{t} \int_{\bbB_{2}^{d} (M)} \kappa \left( \bsx; t, \bftheta \right) \pi \left( \mathrm{d} \bftheta \right) \cdot g (\bsx) \right| \\
        = \ & \left| \int_{\bbB_{2}^{d} (M)} \partial_{t} \kappa \left( \bsx; t, \bftheta \right) \pi \left( \mathrm{d} \bftheta \right) \cdot g (\bsx) \right| \\
        \leq \ & \int_{\bbB_{2}^{d} (M)} \left| \partial_{t} \kappa \left( \bsx; t, \bftheta \right) g (\bsx) \right| \pi \left( \mathrm{d} \bftheta \right) \\
        \leq \ & TM^2 L \cdot \phi_d (\bsx) \exp \left\{ \left( c + TM \right) \left\| \bsx \right\|_{2} \right\} + ML \phi_d (\bsx) 
        \left( 1 + \left\| \bsx \right\|_{2} \right) \exp \left\{ \left( c + TM \right) \left\| \bsx \right\|_{2} \right\}.
    \end{split}
\end{equation}
Since Lemma \ref{lemma:gaussian_domination} implies
\begin{equation}
    \label{eqn:proof_lemma:gaussian_integration_by_parts_v7}
    \begin{split}
        \int_{\bbR^d} \phi_d (\bsx) \exp \left\{ \left( c + TM \right) \left\| \bsx \right\|_{2} \right\} \mathrm{d} \bsx \leq \ & \calS \left( c + TM; d \right), \\
        \int_{\bbR^d} \phi_d (\bsx) \left( 1 + \left\| \bsx \right\|_{2} \right) \exp \left\{ \left( c + TM \right) \left\| \bsx \right\|_{2} \right\} \mathrm{d} \bsx \leq \ & \calK \left( c + TM; d, 1 \right),
    \end{split}
\end{equation}
the inequality \eqref{lemma:gaussian_domination} proves the absolute integrability of the function $\bsx \in \bbR^d \mapsto \partial_{t} f_{t}^{G} (\bsx) (\bsx) g (\bsx) \in \bbR$. Likewise, using the spatial derivative of $\kappa \left( \bsx; t, \bftheta \right)$ from \eqref{eqn:proof_lemma:gaussian_integration_by_parts_v1}, one can find that
\[
    \begin{split}
        \left| \kappa \left( \bsx; t, \bftheta \right) \cdot \nabla_{\bsx}^{\top} g (\bsx) \bsu_{t} (\bsx) \right|
        \stackrel{\textnormal{(b)}}{\leq} \ & M \kappa \left( \bsx; t, \bftheta \right) \left\| \nabla_{\bsx} g (\bsx) \right\|_{2} \\
        \leq \ & ML \phi_d (\bsx) \exp \left\{ \left( c + TM \right) \left\| \bsx \right\|_{2} \right\},
    \end{split}
\]
where the step (b) holds due to the Cauchy-Schwarz inequality with $\left\| \bsu_{t} (\bsx) \right\|_{2} \leq M$. The above bound tells us that
\begin{equation}
    \label{eqn:proof_lemma:gaussian_integration_by_parts_v8}
    \begin{split}
        \left| p_{t} (\bsx) \cdot \nabla_{\bsx}^{\top} g (\bsx) \bsu_{t} (\bsx) \right| = \ & \left| \int_{\bbB_{2}^{d} (M)} \kappa \left( \bsx; t, \bftheta \right) \pi \left( \mathrm{d} \bftheta \right) \cdot \nabla_{\bsx}^{\top} g (\bsx) \bsu_{t} (\bsx) \right| \\
        \leq \ & \int_{\bbB_{2}^{d} (M)} \left| \kappa \left( \bsx; t, \bftheta \right) \cdot \nabla_{\bsx}^{\top} g (\bsx) \bsu_{t} (\bsx) \right| \pi \left( \mathrm{d} \bftheta \right) \\
        \leq \ & ML \phi_d (\bsx) \exp \left\{ \left( c + TM \right) \left\| \bsx \right\|_{2} \right\},
    \end{split}
\end{equation}
which verifies the absolute integrability of the function $\bsx \in \bbR^d \mapsto p_{t} (\bsx) \cdot \nabla_{\bsx}^{\top} g (\bsx) \bsu_{t} (\bsx) \in \bbR$ in conjunction with the first observation from \eqref{eqn:proof_lemma:gaussian_integration_by_parts_v7}.
\medskip

\indent Next, we turn our attention to establishing the Gaussian integration by parts identity \eqref{eqn:lemma:gaussian_integration_by_parts_v2}. In order to justify the integration by parts argument without assuming boundary decay implicitly, we define
\begin{equation}
    \label{eqn:proof_lemma:gaussian_integration_by_parts_v9}
    \begin{split}
        \psi (t) :=
        \begin{cases}
            1 & \textnormal{if } 0 \leq t < 1; \\
            1 - 3 (t-1)^2 + 2 (t-1)^3 & \textnormal{if } 1 \leq t < 2; \\
            0 & \textnormal{if } t \geq 2,
        \end{cases} \qquad
        \chi_{R} (\bsx) := \psi \left( \frac{\left\| \bsx \right\|_{2}}{R} \right), \qquad R \in (0, +\infty).
    \end{split}
\end{equation}
One can easily reveal that the function $\chi_{R} : \bbR^d \to \bbR$ is a compactly supported $\calC^1$-function: it is constant near the origin, so the lack of differentiability of the norm function $\left\| \cdot \right\|_{2}$ at zero becomes immaterial. Since
\[
    \begin{split}
        \left| \psi' (t) \right| = 6 (t-1) \left\{ 1 - (t-1) \right\} \leq \frac{3}{2}, \qquad t \in (1, 2),
    \end{split}
\]
it follows that
\begin{equation}
    \label{eqn:proof_lemma:gaussian_integration_by_parts_v10}
    \begin{split}
        0 \leq \chi_{R} \leq 1, \qquad \chi_{R} (\bsx) = 1 \textnormal{ for } \bsx \in \bbB_{2}^{d} (R), \qquad \chi_{R} (\bsx) = 0 \textnormal{ for } \bsx \in \bbR^d \setminus \bbB_{2}^{d} \left( 2R \right),
    \end{split}
\end{equation}
and
\begin{equation}
    \label{eqn:proof_lemma:gaussian_integration_by_parts_v11}
    \begin{split}
        \left\| \nabla_{\bsx} \chi_{R} (\bsx) \right\|_{2} \leq \frac{3}{2R}, \qquad \bsx \in \bbR^d.
    \end{split}
\end{equation}
Coordinate-wise integration by parts for the compactly supported $\calC^1$-function $g \chi_R$ then leads to
\begin{equation}
    \label{eqn:proof_lemma:gaussian_integration_by_parts_v12}
    \begin{split}
        - \int_{\bbR^d} \bftheta^{\top} \nabla_{\bsx} \kappa \left( \bsx; t, \bftheta \right) g (\bsx) \chi_{R} (\bsx) \mathrm{d} \bsx = \ & \int_{\bbR^d} \kappa \left( \bsx; t, \bftheta \right) \bftheta^{\top} \nabla_{\bsx} \left( g \chi_{R} \right) (\bsx) \mathrm{d} \bsx \\
        = \ & \underbrace{\int_{\bbR^d} \kappa \left( \bsx; t, \bftheta \right) \chi_{R} (\bsx) \bftheta^{\top} \nabla_{\bsx} g (\bsx) \mathrm{d} \bsx}_{=: \ \textnormal{(T1)}} \\
        &+ \underbrace{\int_{\bbR^d} \kappa \left( \bsx; t, \bftheta \right) g (\bsx) \bftheta^{\top} \nabla_{\bsx} \chi_{R} (\bsx) \mathrm{d} \bsx}_{=: \ \textnormal{(T2)}}.
    \end{split}
\end{equation}
Note that the term (T2) has absolute value at most
\[
    \begin{split}
        \left| \textnormal{(T2)} \right| \stackrel{\textnormal{(c)}}{\leq} \frac{3M}{2R} \int_{\bbR^d} \kappa \left( \bsx; t, \bftheta \right) \left| g (\bsx) \right| \mathrm{d} \bsx \stackrel{\textnormal{(d)}}{\leq} \frac{3ML}{2R} \calS \left( c + TM; d \right) \stackrel{R \to \infty}{\longrightarrow} 0,
    \end{split}
\]
where the step (c) follows from the Cauchy-Schwarz inequality along with the fact \eqref{eqn:proof_lemma:gaussian_integration_by_parts_v11}, and the step (d) leverages the inequality \eqref{eqn:proof_lemma:gaussian_integration_by_parts_v5} in conjunction with \eqref{eqn:proof_lemma:gaussian_integration_by_parts_v7}. Invoking the absolute integrability of the functions
\[
    \begin{split}
        \bsx \in \bbR^d \qquad \mapsto & \qquad \bftheta^{\top} \nabla_{\bsx} \kappa \left( \bsx; t, \bftheta \right) g (\bsx) \in \bbR, \\
        \bsx \in \bbR^d \qquad \mapsto & \qquad \kappa \left( \bsx; t, \bftheta \right) \bftheta^{\top} \nabla_{\bsx} g (\bsx) \in \bbR
    \end{split}
\]
for each fixed $\bftheta \in \bbB_{2}^{d} (M)$, one can apply the dominated convergence theorem to the identity \eqref{eqn:proof_lemma:gaussian_integration_by_parts_v12} to conclude in conjunction with the fact $\textnormal{(T2)} \stackrel{R \to \infty}{\longrightarrow} 0$ that
\begin{equation}
    \label{eqn:proof_lemma:gaussian_integration_by_parts_v13}
    \begin{split}
        \int_{\bbR^d} \kappa \left( \bsx; t, \bftheta \right) \bftheta^{\top} \nabla_{\bsx} g (\bsx) \mathrm{d} \bsx = \ & - \int_{\bbR^d} \bftheta^{\top} \nabla_{\bsx} \kappa \left( \bsx; t, \bftheta \right) g (\bsx) \mathrm{d} \bsx \\
        \stackrel{\textnormal{(e)}}{=} \ &  \int_{\bbR^d} \partial_{t} \kappa \left( \bsx; t, \bftheta \right) g (\bsx) \mathrm{d} \bsx,
    \end{split}
\end{equation}
where the step (e) results from the equation \eqref{eqn:proof_lemma:gaussian_integration_by_parts_v2}. Due to the joint integrability of $\partial_{t} \kappa \left( \bsx; t, \bftheta \right) g (\bsx)$ and $\kappa \left( \bsx; t, \bftheta \right) \cdot \nabla_{\bsx}^{\top} g (\bsx) \bftheta$ in the joint variables $\left( \bsx, \bftheta \right)$, the Fubini-Tonelli theorem gives
\[
    \begin{split}
        \int_{\bbR^d} \partial_{t} p_{t} (\bsx) g (\bsx) \mathrm{d} \bsx = \ & \int_{\bbR^d} \left[ \int_{\bbB_{2}^{d} (M)} \partial_{t} \kappa \left( \bsx; t, \bftheta \right) g (\bsx) \pi \left( \mathrm{d} \bftheta \right) \right] \mathrm{d} \bsx \\
        = \ & \int_{\bbB_{2}^{d} (M)} \underbrace{\left[ \int_{\bbR^d} \partial_{t} \kappa \left( \bsx; t, \bftheta \right) g (\bsx) \mathrm{d} \bsx \right]}_{= \ \int_{\bbR^d} \kappa \left( \bsx; t, \bftheta \right) \nabla_{\bsx}^{\top} g (\bsx) \bftheta \mathrm{d} \bsx} \pi \left( \mathrm{d} \bftheta \right) \\
        \stackrel{\textnormal{(f)}}{=} \ & \int_{\bbR^d} \nabla_{\bsx}^{\top} g (\bsx) \underbrace{\left[ \int_{\bbB_{2}^{d} (M)} \bftheta \kappa \left( \bsx; t, \bftheta \right) \pi \left( \mathrm{d} \bftheta \right) \right]}_{= \ p_{t} (\bsx) \cdot \bsu_{t} (\bsx)} \mathrm{d} \bsx \\
        = \ & \int_{\bbR^d} p_{t} (\bsx) \cdot \nabla_{\bsx}^{\top} g (\bsx) \bsu_{t} (\bsx) \mathrm{d} \bsx,
    \end{split}
\]
completing the proof for the desired Gaussian integration by parts identity \eqref{eqn:lemma:gaussian_integration_by_parts_v2}, where the step (f) makes use of the equation \eqref{eqn:proof_lemma:gaussian_integration_by_parts_v13}.

\subsubsection{Likelihood-ratio growth}
\label{subsubsec:likelihood_ratio_growth}

By the part (\romannumeral 1) of Lemma \ref{lemma:interpolation_parameter_differentiation_density_envelopes}, we obtain the following bound on the likelihood ratio: for each fixed $T \in (0, +\infty)$,
\begin{equation}
    \label{eqn:proof_likelihood_ratio_growth_v1}
    \begin{split}
        \exp \left\{ - b_{T} (\bsx) \right\} \leq \rho_{t} (\bsx) \leq \exp \left\{ b_{T} (\bsx) \right\}, \qquad \left( t, \bsx \right) \in [-T, T] \times \bbR^d,
    \end{split}
\end{equation}
where
\[
    \begin{split}
        b_{T} (\bsx) := 2 TM \left\| \bsx \right\|_{2} + \frac{T^2 M^2}{2}, \qquad \bsx \in \bbR^d.
    \end{split}
\]
The inequality \eqref{eqn:proof_likelihood_ratio_growth_v1} can be equivalently interpreted as
\begin{equation}
    \label{eqn:proof_likelihood_ratio_growth_v2}
    \begin{split}
        \left| \log \rho_{t} (\bsx) \right| \leq b_{T} (\bsx), \qquad \left( t, \bsx \right) \in [-T, T] \times \bbR^d,
    \end{split}
\end{equation}

\noindent Let us now corroborate that the functions $\log \rho_t : \bbR^d \to \bbR$ and $\rho_{t}^{\alpha} : \bbR^d \to (0, +\infty)$, $\alpha \in \bbR$, satisfy the exponential growth condition \eqref{eqn:lemma:gaussian_integration_by_parts_v1} via the following lemma:

\begin{lemma}
\label{lemma:likelihood_ratio_exponential_growth_condition}
For every $\left( \alpha, t \right) \in \bbR \times \bbR$, it holds that
\begin{equation}
    \label{eqn:lemma:likelihood_ratio_exponential_growth_condition_v1}
    \begin{split}
        \nabla_{\bsx} \left( \rho_{t}^{\alpha} \right) (\bsx) = \alpha t \rho_{t}^{\alpha} (\bsx) \bfDelta_{t} (\bsx).
    \end{split}
\end{equation}
Now, fix any $T \in (0, +\infty)$. Then, the following conclusions hold:
\begin{enumerate} [label = (\roman*)]
    \item For any fixed $t \in \left[ -T, T \right]$, the function $\rho_{t}^{\alpha}: \bbR^d \to (0, +\infty)$ obeys the exponential growth condition \eqref{eqn:lemma:gaussian_integration_by_parts_v1} with the choice $\left( L, c \right) = \left( \left( 1 + 2 |\alpha| TM \right) \exp \left( \frac{| \alpha | T^2 M^2}{2} \right), 2 |\alpha| TM \right)$, i.e.,
    \begin{equation}
        \label{eqn:lemma:likelihood_ratio_exponential_growth_condition_v2}
        \begin{split}
            \left| \rho_{t}^{\alpha} (\bsx) \right| + \left\| \nabla_{\bsx} \left( \rho_{t}^{\alpha} \right) (\bsx) \right\|_{2} \leq \left( 1 + 2 |\alpha| TM \right) \exp \left( \frac{| \alpha | T^2 M^2}{2} \right) \exp \left( 2 |\alpha| TM \left\| \bsx \right\|_{2} \right), \qquad \bsx \in \bbR^d.
        \end{split}
    \end{equation}
    \item For each fixed $t \in [-T, T]$, the function $\log \rho_t : \bbR^d \to \bbR$ satisfies the exponential growth condition \eqref{eqn:lemma:gaussian_integration_by_parts_v1} with the choice $\left( L, c \right) = \left( 1 + 2 TM, \frac{T^2 M^2}{2}, 1 \right)$, i.e.,
    \begin{equation}
        \label{eqn:lemma:likelihood_ratio_exponential_growth_condition_v3}
        \begin{split}
            \left| \log \rho_{t} (\bsx) \right| + \left\| \nabla_{\bsx} \log \rho_{t} (\bsx) \right\|_{2} \leq \left( 1 + 2 TM + \frac{T^2 M^2}{2} \right) \exp \left( \left\| \bsx \right\|_{2} \right), \qquad \bsx \in \bbR^d.
        \end{split}
    \end{equation}
\end{enumerate}
\end{lemma}

\begin{proof} [Proof of Lemma \ref{lemma:likelihood_ratio_exponential_growth_condition}]
\indent It is straightforward to observe that $
\nabla_{\bsx} (1) (\bsx) = \bfzero_d$, and for any given $\alpha \in \bbR \setminus \{ 0 \}$,
\[
    \begin{split}
        \nabla_{\bsx} \left( \rho_{t}^{\alpha} \right) (\bsx) = \alpha \rho_{t}^{\alpha - 1} (\bsx) \cdot \nabla_{\bsx} \rho_{t} (\bsx) = \alpha \rho_{t}^{\alpha - 1} (\bsx) \cdot \rho_{t} (\bsx) \underbrace{\nabla_{\bsx} \log \rho_{t} (\bsx)}_{= \ t \bfDelta_{t} (\bsx)} \stackrel{\textnormal{(a)}}{=} \alpha t \rho_{t}^{\alpha} (\bsx) \bfDelta_{t} (\bsx),
    \end{split}
\]
where the step (a) follows due to the third piece from \eqref{eqn:lemma:posterior_means_score_identities}. Now, let us focus on establishing the desired exponential growth conditions for the functions $\rho_{t}^{\alpha} : \bbR^d \to (0, +\infty)$ and $\log \rho_t : \bbR^d \to \bbR$:
\medskip

\indent (\romannumeral 1) Utilizing the facts $|t| \leq T$ and $\left\| \bfDelta_{t} (\bsx) \right\|_{2} \leq 2M$ for $\bsx \in \bbR^d$, it follows from the identity \eqref{eqn:lemma:likelihood_ratio_exponential_growth_condition_v1} that
\[
    \begin{split}
        \left| \rho_{t}^{\alpha} (\bsx) \right| + \left\| \nabla_{\bsx} \left( \rho_{t}^{\alpha} \right) (\bsx) \right\|_{2} \leq \ & \left( 1 + 2 \left| \alpha \right| TM \right) \rho_{t}^{\alpha} (\bsx) \\
        \stackrel{\textnormal{(b)}}{\leq} \ & \left( 1 + 2 \left| \alpha \right| TM \right) \exp \left\{ \left| \alpha \right| b_{T} (\bsx) \right\} \\
        = \ & \left( 1 + 2 |\alpha| TM \right) \exp \left( \frac{| \alpha | T^2 M^2}{2} \right) \exp \left( 2 |\alpha| TM \left\| \bsx \right\|_{2} \right)
    \end{split}
\]
for all $\bsx \in \bbR^d$, where the step (b) is a consequence of the bound \eqref{eqn:proof_likelihood_ratio_growth_v1} on the likelihood ratio. This completes the proof of the part (\romannumeral 1) from Lemma \ref{lemma:likelihood_ratio_exponential_growth_condition}.
\medskip

\indent (\romannumeral 2) By the third piece from \eqref{eqn:lemma:posterior_means_score_identities} alongside with the facts $|t| \leq T$ and $\left\| \bfDelta_{t} (\bsx) \right\|_{2} \leq 2M$ for $\bsx \in \bbR^d$, we obtain
\[
    \begin{split}
        \left| \log \rho_{t} (\bsx) \right| + \left\| \nabla_{\bsx} \log \rho_{t} (\bsx) \right\|_{2} \stackrel{\textnormal{(c)}}{\leq} \ & b_{T} (\bsx) + 2TM \\
        \leq \ & \left( 1 + 2TM + \frac{T^2 M^2}{2} \right) \left( 1 + \left\| \bsx \right\|_{2} \right) \\
        \leq \ & \left( 1 + 2TM + \frac{T^2 M^2}{2} \right) \exp \left( \left\| \bsx \right\|_{2} \right)
    \end{split}
\]
for every $\bsx \in \bbR^d$, where the step (c) comes from the log-likelihood ratio bound \eqref{eqn:proof_likelihood_ratio_growth_v2}, concluding the proof of the part (\romannumeral 2) of Lemma \ref{lemma:likelihood_ratio_exponential_growth_condition}.
\end{proof}

\subsubsection{Finiteness and regularity of the divergence curves}
\label{subsubsec:finiteness_regularity_divergence_curves}

\paragraph{Differentiation of the normalization.}
It holds that
\begin{equation}
    \label{eqn:proof_finiteness_regularity_divergence_curves_v1}
    \begin{split}
        \int_{\bbR^d} \partial_{t} p_{t} (\bsx) \mathrm{d} \bsx \stackrel{\textnormal{(a)}}{=} \partial_{t} \underbrace{\int_{\bbR^d} p_{t} (\bsx) \mathrm{d} \bsx}_{= \ 1} = 0, \qquad \int_{\bbR^d} \partial_{t} q_{t} (\bsx) \mathrm{d} \bsx \stackrel{\textnormal{(b)}}{=} \partial_{t} \underbrace{\int_{\bbR^d} q_{t} (\bsx) \mathrm{d} \bsx}_{= \ 1} = 0.
    \end{split}
\end{equation}
In order to derive the identity \eqref{eqn:proof_finiteness_regularity_divergence_curves_v1}, it suffices to provide an analytic justification for interchanging the order of the differentiation and integration for the steps (a) and (b). Given any compact interval $K \subseteq \bbR$, choose a sufficiently large $T \in (0, +\infty)$ so that $K \subseteq [-T, T]$. Combining two pieces \eqref{eqn:proof_lemma:interpolation_parameter_differentiation_density_envelopes_v1} and \eqref{eqn:proof_lemma:interpolation_parameter_differentiation_density_envelopes_v2} from Appendix \ref{subsubsec:density_envelopes_differentiation} together, one has
\begin{equation}
    \label{eqn:proof_finiteness_regularity_divergence_curves_v2}
    \begin{split}
        \phi_d \left( \bsx - t \bftheta \right) \leq \phi_d (\bsx) \exp \left( TM \left\| \bsx \right\|_{2} \right), \qquad \left( t, \bsx \right) \in K \times \bbR^d.
    \end{split}
\end{equation}
Also, by using the fact $\partial_t \phi_d \left( \bsx - t \bftheta \right) = \left( \bftheta^{\top} \bsx - t \left\| \bftheta \right\|_{2}^2 \right) \phi_d \left( \bsx - t \bftheta \right)$, the inequality \eqref{eqn:proof_finiteness_regularity_divergence_curves_v2} implies
\begin{equation}
    \label{eqn:proof_finiteness_regularity_divergence_curves_v3}
    \begin{split}
        \left| \partial_{t} \phi_d \left( \bsx - t \bftheta \right) \right| \leq \left( M \left\| \bsx \right\|_{2} + T M^2 \right) \phi_d (\bsx) \exp \left( TM \left\| \bsx \right\|_{2} \right), \qquad \left( t, \bsx \right) \in K \times \bbR^d.
    \end{split}
\end{equation}
Since Lemma \ref{lemma:gaussian_domination} gives
\begin{equation}
    \label{eqn:proof_finiteness_regularity_divergence_curves_v4}
    \begin{split}
        \int_{\bbR^d} \phi_d (\bsx) \exp \left( TM \left\| \bsx \right\|_{2} \right) \mathrm{d} \bsx \leq \ & \calS \left( TM; d \right), \\
        \int_{\bbR^d} \phi_d (\bsx) \left( 1 + \left\| \bsx \right\|_{2} \right) \exp \left( TM \left\| \bsx \right\|_{2} \right) \mathrm{d} \bsx \leq \ & \calK \left( TM; d, 1 \right), \\
        \int_{\bbR^d} \phi_d (\bsx) \left( 1 + \left\| \bsx \right\|_{2} \right)^2 \exp \left( TM \left\| \bsx \right\|_{2} \right) \mathrm{d} \bsx \leq \ & \calK \left( TM; d, 2 \right),
    \end{split}
\end{equation}
we find that the right-hand sides from the inequalities \eqref{eqn:proof_finiteness_regularity_divergence_curves_v2} and \eqref{eqn:proof_finiteness_regularity_divergence_curves_v3} are $\mathrm{Leb}_{d}$-integrable functions with respect to $\bsx$ and they are independent of $\left( t, \bftheta \right)$. Thus, one can invoke Lemma \ref{lemma:dominated_differentiation} to justify the steps (a) and (b), completing the proof of the desired identity \eqref{eqn:proof_finiteness_regularity_divergence_curves_v1}.

\paragraph{Finiteness and $\calC^1$-regularity of $J: \bbR \to [0, +\infty]$.}
To establish the finiteness of the function $J: \bbR \to [0, +\infty]$, it suffices to show that
\begin{equation}
    \label{eqn:proof_finiteness_regularity_divergence_curves_v5}
    \begin{split}
        \int_{\bbR^d} \left| \calJ_{0} \left( t, \bsx \right) \right| \mathrm{d} \bsx < +\infty, \qquad \int_{\bbR^d} \left| \calJ_{1} \left( t, \bsx \right) \right| \mathrm{d} \bsx < +\infty,
    \end{split}
\end{equation}
for every fixed $t \in \bbR$. Fix any $t \in \bbR$ and choose $T \in (0, +\infty)$ so that $|t| < T$. Then, the bound \eqref{eqn:proof_likelihood_ratio_growth_v2} and the part (\romannumeral 1) of Lemma \ref{lemma:interpolation_parameter_differentiation_density_envelopes} yields
\begin{equation}
    \label{eqn:proof_finiteness_regularity_divergence_curves_v6}
    \begin{split}
        \left| \calJ_{0} \left( t, \bsx \right) \right| \leq \ & b_{T} (\bsx) p_t (\bsx) \\
        \leq \ & b_{T} (\bsx) \cdot \phi_{d} (\bsx) \exp \left( TM \left\| \bsx \right\|_{2} \right) \\
        \leq \ & \frac{T^2 M^2}{2} \phi_{d} (\bsx) \exp \left( TM \left\| \bsx \right\|_{2} \right) + 2 TM \phi_{d} (\bsx) \left( 1 + \left\| \bsx \right\|_{2} \right) \exp \left( TM \left\| \bsx \right\|_{2} \right).
    \end{split}
\end{equation}
The observations from \eqref{eqn:proof_finiteness_regularity_divergence_curves_v4} imply that the dominating integrands in \eqref{eqn:proof_finiteness_regularity_divergence_curves_v6} are $\mathrm{Leb}_{d}$-integrable, thereby
\[
    \begin{split}
        \int_{\bbR^d} \left| \calJ_{0} \left( t, \bsx \right) \right| \mathrm{d} \bsx < +\infty
    \end{split}
\]
as desired in the first claim of \eqref{eqn:proof_finiteness_regularity_divergence_curves_v5}. We can prove the remaining claim from \eqref{eqn:proof_finiteness_regularity_divergence_curves_v5} by applying exactly the same arguments.

\indent We next establish the $\calC^{1}$-regularity of $J: \bbR \to [0, +\infty)$. To this end, let us consider any compact interval $K \subseteq \bbR$ and choose a sufficiently large $T \in (0, +\infty)$ so that $K \subseteq [-T, T]$. Then, the bound \eqref{eqn:proof_finiteness_regularity_divergence_curves_v6} directly yields
\[
    \begin{split}
        \left| \calJ_{0} \left( t, \bsx \right) \right| \leq P_{\calJ_0}^{K} (\bsx) := \frac{T^2 M^2}{2} \phi_{d} (\bsx) \exp \left( TM \left\| \bsx \right\|_{2} \right) + 2 TM \phi_{d} (\bsx) \left( 1 + \left\| \bsx \right\|_{2} \right) \exp \left( TM \left\| \bsx \right\|_{2} \right) 
    \end{split}
\]
for all $\left( t, \bsx \right) \in K \times \bbR^d$. We now inspect the derivative of $H_{0} \left( t, \bsx \right)$ with respect to the interpolation parameter. One can easily reveal that
\begin{equation}
    \label{eqn:proof_finiteness_regularity_divergence_curves_v7}
    \begin{split}
        \partial_t \calJ_{0} \left( t, \bsx \right) = \partial_{t} p_t (\bsx) \left\{ 1 + \log \rho_{t} (\bsx) \right\} - \rho_{t} (\bsx) \partial_{t} q_t (\bsx).
    \end{split}
\end{equation}
Therefore, it follows that
\[
    \begin{split}
        \left| \partial_t \calJ_{0} \left( t, \bsx \right) \right| \stackrel{\textnormal{(c)}}{\leq} \ & \left| \partial_{t} p_t (\bsx) \right| \left\{ 1 + b_{T} (\bsx) \right\} + b_{T} (\bsx) \left| \partial_{t} q_t (\bsx) \right| \\
        \stackrel{\textnormal{(d)}}{\leq} \ & a_{T} (\bsx) \left\{ 1 + b_{T} (\bsx) \right\} p_t (\bsx) + a_{T} (\bsx) b_{T} (\bsx) q_t (\bsx) \\
        \stackrel{\textnormal{(e)}}{\leq} \ & 2 a_{T} (\bsx) \left\{ 1 + b_{T} (\bsx) \right\} \phi_{d} (\bsx) \exp \left( TM \left\| \bsx \right\|_{2} \right) \\
        \leq \ & Q_{\calJ_0}^{K} (\bsx), 
    \end{split}
\]
where
\[
    \begin{split}
        Q_{\calJ_0}^{K} (\bsx) := \ & T M^2 \left( 1 + \frac{T^2 M^2}{2} \right) \phi_{d} (\bsx) \exp \left( TM \left\| \bsx \right\|_{2} \right) \\
        &+ \left\{ 2 T^2 M^3 + M \left( 1 + \frac{T^2 M^2}{2} \right) \right\} \phi_{d} (\bsx) \left( 1 + \left\| \bsx \right\|_{2} \right) \exp \left( TM \left\| \bsx \right\|_{2} \right) \\
        &+ 2 T M^2 \phi_{d} (\bsx) \left( 1 + \left\| \bsx \right\|_{2} \right)^2 \exp \left( TM \left\| \bsx \right\|_{2} \right),
    \end{split}
\]
the step (c) results from the log-likelihood ratio bound \eqref{eqn:proof_likelihood_ratio_growth_v2}, the step (d) makes use of the inequality \eqref{eqn:lemma:interpolation_parameter_differentiation_density_envelopes_v4} from Lemma \ref{lemma:interpolation_parameter_differentiation_density_envelopes}, and the step (e) holds due to the part (\romannumeral 1) of Lemma \ref{lemma:interpolation_parameter_differentiation_density_envelopes}. Here, it follows from the observations from \eqref{eqn:proof_finiteness_regularity_divergence_curves_v4} that
\[
    \begin{split}
        \int_{\bbR^d} \left\{ P_{\calJ_0}^{K} (\bsx) + Q_{\calJ_0}^{K} (\bsx) \right\} \mathrm{d} \bsx < +\infty.
    \end{split}
\]
Hence, one can conclude that the function $t \in \bbR \mapsto \textsf{KL} \left( p_t \left\| q_t \right. \right) \in [0, +\infty)$ belongs to the class $\calC^{1} (\bbR)$ by virtue of Lemma \ref{lemma:dominated_differentiation}. Using exactly the same arguments, it is easy to reveal that the function $t \in \bbR \mapsto \textsf{KL} \left( q_t \left\| p_t \right. \right) \in [0, +\infty)$ also belongs to the class $\calC^{1} (\bbR)$, which proves the $\calC^{1}$-regularity of the function $J: \bbR \to [0, +\infty)$.

\paragraph{Finiteness and $\calC^1$-regularity of $C: \bbR \to [0, +\infty]$.}
To verify finiteness, fix any $t \in \bbR$ and select $T \in (0, +\infty)$ such that $|t| < T$. For the integrand defined above, we have
\begin{equation}
    \label{eqn:proof_finiteness_regularity_divergence_curves_v8}
    \begin{split}
       \calC \left( t, \bsx \right) = \ & p_t (\bsx) \cdot \rho_{t} (\bsx) \\
        \stackrel{\textnormal{(f)}}{\leq} \ & \phi_d (\bsx) \exp \left( TM \left\| \bsx \right\|_{2} \right) \cdot \exp \left\{ b_{T} (\bsx) \right\} \\
        = \ & \exp \left( \frac{T^2 M^2}{2} \right) \phi_d (\bsx) \exp \left( 3 TM \left\| \bsx \right\|_{2} \right)
    \end{split}
\end{equation}
for every $\bsx \in \bbR^d$, where the step (f) invokes the part (\romannumeral 1) of Lemma \ref{lemma:interpolation_parameter_differentiation_density_envelopes} as well as the likelihood ratio bound \eqref{eqn:proof_likelihood_ratio_growth_v1}. Therefore, we have
\[
    \begin{split}
        0 \leq \ & C (t) = \int_{\bbR^d} \calC \left( t, \bsx \right) \mathrm{d} \bsx - 1 \\
        \leq \ & \exp \left( \frac{T^2 M^2}{2} \right) \int_{\bbR^d} \phi_d (\bsx) \exp \left( 3 TM \left\| \bsx \right\|_{2} \right) \mathrm{d} \bsx - 1 \\
        \stackrel{\textnormal{(g)}}{\leq} \ & \exp \left( \frac{T^2 M^2}{2} \right) \calS \left( 3 TM; d \right) - 1 < +\infty,
    \end{split}
\]
demonstrating the finiteness of the function $C$, where the step (g) results from Lemma \ref{lemma:gaussian_domination}.

\indent Now, it is time to establish the $\calC^1$-regularity of the function $C : \bbR \to [0, +\infty)$. Towards this end, let us consider any compact interval $K \subseteq \bbR$ and choose a sufficiently large $T \in (0, +\infty)$ satisfying $K \subseteq [-T, T]$. By using the same arguments as those in the bound \eqref{eqn:proof_finiteness_regularity_divergence_curves_v8}, one has
\[
    \begin{split}
        \left| \calC \left( t, \bsx \right) \right| \leq P_{\calC}^{K} (\bsx) := \exp \left( \frac{T^2 M^2}{2} \right) \phi_d (\bsx) \exp \left( 3 TM \left\| \bsx \right\|_{2} \right), \qquad \left( t, \bsx \right) \in K \times \bbR^d.
    \end{split}
\]
Thus, it follows for every $\left( t, \bsx \right) \in K \times \bbR^d$ that
\[
    \begin{split}
        \left| \partial_{t} \calC \left( t, \bsx \right) \right| \leq \ & 2 \left| \partial_{t} p_t (\bsx) \right| \cdot \rho_{t} (\bsx) + \left|  \partial_{t} q_t (\bsx) \right| \cdot \rho_{t}^2 (\bsx) \\
        \stackrel{\textnormal{(h)}}{\leq} \ & 2 a_{T} (\bsx) p_t (\bsx) \rho_{t} (\bsx) + a_{T} (\bsx) q_t (\bsx) \rho_{t}^2 (\bsx) \\
        \stackrel{\textnormal{(i)}}{\leq} \ & 2 a_{T} (\bsx) \phi_{d} (\bsx) \exp \left( TM \left\| \bsx \right\|_{2} \right) \cdot \exp \left\{ b_{T} (\bsx) \right\} + a_{T} (\bsx) \phi_{d} (\bsx) \exp \left( TM \left\| \bsx \right\|_{2} \right) \cdot \exp \left\{ 2 b_{T} (\bsx) \right\} \\
        \leq \ & 4 a_{T} (\bsx) \phi_{d} (\bsx) \exp \left( TM \left\| \bsx \right\|_{2} \right) \cdot \exp \left\{ 2 b_{T} (\bsx) \right\} \\
        \leq \ & Q_{\calC}^{K} (\bsx),
    \end{split}
\]
where
\[
    \begin{split}
        Q_{\calC}^{K} (\bsx) := 4 T M^2 \exp \left( T^2 M^2 \right) \phi_{d} (\bsx) \exp \left( 5 TM \left\| \bsx \right\|_{2} \right) + 4 M \exp \left( T^2 M^2 \right) \phi_{d} (\bsx) \left( 1 + \left\| \bsx \right\|_{2} \right) \exp \left( 5 TM \left\| \bsx \right\|_{2} \right).
    \end{split}
\]
Here, the steps (h) and (i) are consequences from the inequality \eqref{eqn:lemma:interpolation_parameter_differentiation_density_envelopes_v4} and the part (\romannumeral 1) from Lemma \ref{lemma:interpolation_parameter_differentiation_density_envelopes}, respectively. Because Lemma \ref{lemma:gaussian_domination} yields the absolute integrability of the functions $P_{\calC}^{K}: \bbR^d \to [0, +\infty)$ and $Q_{\calC}^{K}: \bbR^d \to [0, +\infty)$:
\[
    \begin{split}
        &\int_{\bbR^d} \left\{ P_{\calC}^{K} (\bsx) + Q_{\calC}^{K} (\bsx) \right\} \mathrm{d} \bsx \\
        \leq \ & \exp \left( \frac{T^2 M^2}{2} \right) \calS \left( 3 TM; d \right) + 4 TM^2 \exp \left( T^2 M^2 \right) \calS \left( 5 TM; d \right) + 4M \exp \left( T^2 M^2 \right) \calK \left( 5 TM; d, 1 \right) < +\infty,
    \end{split}
\]
we obtain the desired $\calC^1$-regularity of the function $C : \bbR \to [0, +\infty)$ by applying Lemma \ref{lemma:dominated_differentiation}.

\paragraph{$\calC^1$-regularity of the Hellinger affinity and $D : \bbR \to [0, 2]$.}
Now, fix any compact interval $K \subseteq \bbR$ and select a sufficiently large $T \in (0, +\infty)$ with $K \subseteq [-T, T]$. One can easily observe from the part (\romannumeral 1) of Lemma \ref{lemma:interpolation_parameter_differentiation_density_envelopes} that
\[
    \begin{split}
        \left| \calA \left( t, \bsx \right) \right| \leq P_{\calA}^{K} (\bsx) := \phi_d (\bsx) \exp \left( TM \left\| \bsx \right\|_{2} \right), \qquad \left( t, \bsx \right) \in K \times \bbR^d.
    \end{split}
\]
Using the pointwise derivative computed above, we obtain a $t$-uniform upper bound on $\left| \partial_{t} \calA \left( t, \bsx \right) \right|$:
\[
    \begin{split}
        \left| \partial_{t} \calA \left( t, \bsx \right) \right| \leq \ & \frac{1}{2 \sqrt{\rho_{t} (\bsx)}} \left| \partial_{t} p_t (\bsx) \right| + \frac{\sqrt{\rho_{t} (\bsx)}}{2} \left| \partial_{t} q_t (\bsx) \right| \\
        \stackrel{\textnormal{(j)}}{\leq} \ & \frac{1}{2 \sqrt{\rho_{t} (\bsx)}} \cdot a_{T} (\bsx) p_t (\bsx) + \frac{\sqrt{\rho_{t} (\bsx)}}{2} \cdot a_{T} (\bsx) q_t (\bsx) \\
        \stackrel{\textnormal{(k)}}{\leq} \ & \frac{1}{2} a_{T} (\bsx) \phi_{d} (\bsx) \exp \left( TM \left\| \bsx \right\|_{2} \right) \left\{ \sqrt{\rho_{t} (\bsx)} + \frac{1}{\sqrt{\rho_{t} (\bsx)}} \right\} \\
        \stackrel{\textnormal{(l)}}{\leq} \ & a_{T} (\bsx) \exp \left\{ \frac{1}{2} b_{T} (\bsx) \right\} \phi_{d} (\bsx) \exp \left( TM \left\| \bsx \right\|_{2} \right) \\
        \leq \ & Q_{\calA}^{K} (\bsx),
    \end{split}
\]
where
\[
    \begin{split}
        Q_{\calA}^{K} (\bsx) := \ & T M^2 \exp \left( \frac{T^2 M^2}{4} \right) \phi_d (\bsx) \exp \left( 2 TM \left\| \bsx \right\|_{2} \right) \\
        &+ M \exp \left( \frac{T^2 M^2}{4} \right) \phi_d (\bsx) \left( 1 + \left\| \bsx \right\|_{2} \right) \exp \left( 2 TM \left\| \bsx \right\|_{2} \right).
    \end{split}
\]
Here, the steps (j) and (k) leverage the inequality \eqref{eqn:lemma:interpolation_parameter_differentiation_density_envelopes_v4} and the part (\romannumeral 1) from Lemma \ref{lemma:interpolation_parameter_differentiation_density_envelopes}, respectively, and the step (l) follows by the bound \eqref{eqn:proof_likelihood_ratio_growth_v1} on the likelihood ratio. Since Lemma \ref{lemma:dominated_differentiation} imply
\[
    \begin{split}
        &\int_{\bbR^d} \left\{ P_{\calA}^{K} (\bsx) + Q_{\calA}^{K} (\bsx) \right\} \mathrm{d} \bsx \\
        \leq \ & \calS \left( TM; d \right) + T M^2 \exp \left( \frac{T^2 M^2}{4} \right) \calS \left( 2 TM; d \right) + M \exp \left( \frac{T^2 M^2}{4} \right) \calK \left( 2 TM; d , 1 \right) < +\infty,
    \end{split}
\]
we can apply Lemma \ref{lemma:dominated_differentiation} to conclude that $A \in \calC^1 (\bbR)$. From the relation $D = 2 - 2A$, the $\calC^1$-regularity of $D: \bbR \to [0, 2]$ immediately follows from the $\calC^1$-regularity of $A: \bbR \to [0, 1]$. This completes the analytic justifications for Appendix \ref{subsec:proof_divergence_derivative_formulas}.

\end{document}